\documentclass[oneside, 8pt]{amsart}
\pdfoutput=1
\usepackage{amscd, amsmath, amssymb, amsthm, amsfonts, amstext, verbatim, mathtools, xfrac, microtype, nameref, thmtools}
\usepackage[usenames,dvipsnames]{xcolor}
\usepackage[hyperref=true, backend=biber, bibencoding=utf8, giveninits=true, citestyle=numeric-comp, sortlocale=en_US, url=false, doi=false, eprint=true, maxbibnames=4]{biblatex}
\usepackage[breaklinks=true,unicode]{hyperref}
\definecolor{DarkPurple}{rgb}{0.40,0.0,0.20}
\hypersetup{colorlinks=true,citecolor=NavyBlue,linkcolor=DarkPurple,urlcolor=NavyBlue}
\usepackage{color} 
\usepackage[capitalize]{cleveref}
\usepackage{tikz,tikz-cd}
\usepackage{enumitem}

\renewbibmacro*{volume+number+eid}{\ifentrytype{article}{\- \iffieldundef{volume}{}{Vol.~\printfield{volume},}\iffieldundef{number}{}{ No.~\printfield{number},}}}
\renewbibmacro{in:}{\ifentrytype{article}{}{\printtext{\bibstring{in}\intitlepunct}}}
\newbibmacro{string+doi}[1]{\iffieldundef{doi}{\iffieldundef{url}{#1}{\href{\thefield{url}}{#1}}}{\href{https://dx.doi.org/\thefield{doi}}{#1}}}
\DeclareFieldFormat[article, inproceedings, inbook, book, online]{title}{\usebibmacro{string+doi}{\mkbibquote{#1}}}

\newtheorem{theorem}{Theorem}
\newtheorem*{theorem*}{Theorem}
\newtheorem{lemma}{Lemma}
\newtheorem{prop}[lemma]{Proposition}
\newtheorem{corollary}[lemma]{Corollary}
\newtheorem{externaltheorem}[lemma]{Theorem}
\theoremstyle{remark}

\theoremstyle{definition}
\numberwithin{lemma}{section}
\numberwithin{prop}{section}
\numberwithin{corollary}{section}
\numberwithin{externaltheorem}{section}
\newtheorem{df}[lemma]{Definition} \Crefname{df}{Definition}{Definitions}
 \Crefname{example}{Example}{Examples}
\newtheorem{rem}[lemma]{Remark}

\DeclareMathOperator{\Ker}{Ker}
\DeclareMathOperator{\Img}{Im}
\DeclareMathOperator{\St}{St}
\DeclareMathOperator{\E}{E}
\DeclareMathOperator{\HH}{H}
\DeclareMathOperator{\K}{K}
\DeclareMathOperator{\GG}{G}
\DeclareMathOperator{\KO}{KO}
\DeclareMathOperator{\GW}{GW}

\newcommand{\inv}{^{-1}}

\newcommand{\catname}[1]{{\normalfont\textbf{#1}}} 
\newcommand{\myol}[2][3]{{}\mkern#1mu\overline{\mkern-#1mu#2}}

\newcommand{\ZZ}{\mathbb{Z}}
\newcommand{\rA}{\mathsf{A}}
\newcommand{\rB}{\mathsf{B}}
\newcommand{\rC}{\mathsf{C}}
\newcommand{\rD}{\mathsf{D}}
\newcommand{\rE}{\mathsf{E}}

\newcommand{\rG}{\mathsf{G}}
\tikzcdset{scale cd/.style={every label/.append style={scale=#1},
    cells={nodes={scale=#1}}}}

\numberwithin{equation}{section}

\title{A Horrocks-type theorem for even orthogonal $K_2$}
\author{Andrei Lavrenov}
\email{avlavrenov at gmail.com}
\author {Sergey Sinchuk}
\email {sinchukss at gmail.com}

\address{Chebyshev Laboratory, St. Petersburg State University, 14th. Line V.O. 29b, 199178, St. Petersburg, Russia}

\date {\today}

\begin{document}
\begin{abstract} We prove the Horrocks theorem for unstable even-dimensional orthogonal Steinberg groups. The Horrocks theorem for Steinberg groups is one of the principal ingredients needed for the proof of the $\K_2$-analogue of Serre's problem,  whose positive solution is currently known only in the linear case. \end{abstract}
\maketitle
\section{Introduction}
Recall that the classical Serre problem on projective modules asks if any projective module over a polynomial ring $R = k[x_1,\ldots, x_n]$ over a field $k$ is free. This problem was positively settled by D.~Quillen~\cite{Qu76} and A.~Suslin~\cite{Su76}, and its solution played an important role in the development of algebraic K-theory. We also refer the reader to the textbook~\cite{Lam10} for a comprehensive account on the problem, its history and the subsequent solution.

After the original Serre problem had been solved, numerous analogous questions drew the attention of specialists (see e.\,g.~\cite{Su77, Su82, Abe83, Tu83, Lam10, St-poly, St-Ded}). For example, A.~Suslin formulated and solved the so-called $\K_1$-analogue of Serre's problem. This result asserts that the functor $\K_1(n, R) = \mathrm{GL}_{n}(R)/\E_n(R)$ has the property $\K_1(n, k[x_1, \ldots x_n]) = \K_1(n, k) = k^\times$ for all fields $k$ and $n \geq 3$ see~\cite[Corollary~7.11]{Su77}. Suslin's results were subsequently generalized to $\K_1$-functors modeled on other linear groups (see the definition below). For example, for even-dimensional orthogonal groups the corresponding result was obtained by A.~Suslin and V. Kopeiko in~\cite{Su82}, while for more general types of Chevalley groups of rank $\geq 2$ this is a result of E.~Abe, see~\cite{Abe83}. Recently A.~Stavrova has obtained probably the most general results in this direction: she solved the analogue of Serre problem for the functor $\K_1^G$ modeled on an arbitrary isotropic reductive group scheme $G$ of isotropic rank $\geq 2$ over a field (see~\cite[Theorem~1.2]{St-poly}) and also generalized Abe's result to Dedekind domains (see~\cite[Corollary~1.2]{St-Ded}).

Recall that to every irreducible root system $\Phi$ and a commutative ring $R$ one can associate two groups: the {\it simply-connected Chevalley group} $\GG(\Phi, R)$ (see e.\,g.~\cite[\S~3]{St71} or~\cite{VP}) and the {\it Steinberg group} $\St(\Phi, R)$ (see~\cref{sec:Steinberg-intro} for the definition). There is a well-defined homomorphism $\pi \colon \St(\Phi, R) \to \GG(\Phi, R)$ sending each generator $x_\alpha(\xi)$ to the elementary root unipotent $t_\alpha(\xi)$. The cokernel and the kernel of this homomorphism are denoted $\K_1(\Phi, R)$ and $\K_2(\Phi, R)$, respectively. The latter groups are functorial in $R$. The functors $\K_1(\Phi, -)$ and $\K_2(\Phi, -)$ are called {\it the $\mathrm{K}_1$ and $\mathrm{K_2}$-functors modeled on the Chevalley group $\GG(\Phi, -)$}, see~\cite{St78}.

It turns out that an assertion similar to the Serre problem also holds for the functor $\K_2$. More precisely, in~\cite{Tu83} M. Tulenbaev demonstrated an ``early stability theorem'' from which the isomorphism $\K_2(\rA_\ell, k[x_1, \ldots x_n]) \cong \K_2(\rA_\ell, k) = \K^\mathrm{M}_2(k)$ follows for $\ell \geq 4$. Notice that $\K_2(\rA_\ell, R)$ here is just another notation for the unstable linear functor $\K_2(\ell+1, R)$.

While numerous results on the $\K_1$-analogue of Serre's problem have appeared in the literature since~\cite{Su77} (see e.\,g.~\cite{Su82, Abe83, St-poly, St-Ded}), little progress has been made on the $\K_2$-analogue. It has been conjectured by M.~Wendt, see~\cite[Vermutung~6.22]{Vo11} that a $\K_2$-analogue of the Serre problem holds for $\K_2(\Phi, -)$ for all $\Phi$ of rank $\geq 3$, however this conjecture still remains open for $\Phi$ different from $\rA_\ell$, $\ell \geq 4$.

In~\cite{LS17,La18} the authors have shown that the Steinberg groups $\St(\Phi, R)$ satisfy the Quillen--Suslin local-global principle provided $\Phi$ has rank $\geq 3$ and is either simply-laced or has type $\rC_\ell$. The local-global principle is one of the ingredients needed in the proof of the $\K_2$-analogue of the Serre problem for Chevalley groups. The aim of the present article is to make yet another step towards the solution of the problem, namely to prove an analogue of Horrocks theorem~\cite{Ho64} for Steinberg groups of type $\rD_\ell$.

Our main result is, thus, the following theorem, which is the orthogonal analogue of~\cite[Theorem~5.1]{Tu83} and the $\K_2$-analogue of~\cite[Theorem~6.8]{Su82} (cf. also with~\cite[Theorem~VI.5.2]{Lam10} and~\cite[Theorem~1.1]{St-poly}).
\begin{theorem}[Horrocks theorem for orthogonal $\K_2$]\label{thm:main} Let $A$ be a commutative ring in which $2$ is invertible. Then for any $\ell \geq 7$ the following commutative square is a pullback square in which all homomorphisms are injective:
\[\begin{tikzcd} \KO_2(2\ell, A) \arrow{r} \arrow{d} & \KO_2(2\ell, A[X]) \arrow{d} \\ \KO_2(2\ell, A[X\inv]) \arrow{r} & \KO_2(2\ell, A[X, X\inv]). \end{tikzcd}\]
Moreover, the same assertion holds if one replaces the functor $\KO_2(2\ell, -)$ with $\K_2(\rD_\ell, -)$ or $\St(\rD_\ell, -)$. \end{theorem}
In the above statement $\KO_2(2\ell, -)$ denotes the unstable orthogonal $\K_2$-functor (see~\cref{sec:quillen}). Notice also that in the special case when $A=k$ is a field the assertion of the above theorem is a consequence of the results of U.~Rehmann and J.~Hurrelbrink (see~\cref{field-injectivity} below).

The proof of~\cref{thm:main} goes as follows. We notice that it suffices to prove the $\St(\rD_\ell, -)$-variant of the theorem. Moreover, the proof of the injectivity of
$j_- \colon \St(\rD_\ell, A[X\inv]) \to \St(\rD_\ell, A[X, X\inv])$ turns out to be the hardest part. After invoking the local-global principle~\cite[Theorem~2]{LS17} the proof reduces to the special case when $A$ is local.
Now if $M$ is the maximal ideal of $A$, the proof of the injectivity of $j_-$ comes down to proving the injectivity of the following two homomorphisms:
\[\begin{tikzcd} \St(\rD_\ell, A[X^{-1}]) \arrow{r}{j_B^-} & \St(\rD_\ell, B) \arrow{r}{j_R} & \St(\rD_\ell, R), \end{tikzcd} \]
where $B = A[X\inv] + M[X]$ and $R = A[X, X\inv]$.
The injectivity of $j_R$ is obtained in~\cref{thm41}. This step of the proof relies on the theory of relative central extensions developed by J.-L.~Loday in~\cite{Lo78}. According to this theory, the relative Steinberg group $\St(\Phi, R, I)$ is a central extension of $\Ker(\St(\Phi, R) \to \St(\Phi, R/I))$ by a certain subgroup, which we denote $C(\Phi, R, I)$. It turns out that a simple diagram chasing argument combined with a certain lifting property of relative Steinberg groups (see~\cref{lem:lemma32}) reduces the sought injectivity of $j_R$ to the surjectivity of the map \[C(\rD_\ell, B, M[X, X\inv]) \to C(\rD_\ell, R, M[X, X\inv]).\] This surjectivity is then obtained as a corollary of Panin's stability theorem for orthogonal K-theory and the Bass Fundamental Theorem for higher Grothendieck--Witt groups (these groups include the stable orthogonal K-groups as a special case, see~\eqref{GW-concrete}). This is the only step of our proof which invokes the assumption that $2$ is invertible. 

Finally, the injectivity of the homomorphism $j_B^-$ is obtained in~\cref{thm:P1glueing}, which is a direct generalization of~\cite[Proposition~4.3]{Tu83}. This part of the proof is obtained in somewhat greater generality and is applicable to all simply-laced root systems $\Phi$ containing a subsystem of type $\rA_4$. The main idea of the proof of~\cref{thm:P1glueing} is to express the underlying set of $\St(\Phi, B)$ as an explicit quotient of the Cartesian product of three other groups, one of which is precisely $\St(\Phi, A[X\inv])$, see~\eqref{VT-def}. This can be considered as some sort of a decomposition for $\St(\Phi, B)$. The technique of such decompositions probably dates back to~\cite{ST76}.
The proof of~\cref{thm:P1glueing} is based on a number of preparatory statements. One of these is a certain presentation theorem for Steinberg groups over graded rings similar to the results of U.~Rehmann, C.~Soul{\'e} and M.~Tulenbaev, see~\cref{lemma33} (cf.~\cite[Satz~2]{Re75},\cite[Theorem~2]{RS76},\cite[Lemma~3.3]{Tu83}).
\subsection{Acknowledgements} 
The work of the first-named author (\S~3) was supported by the Russian Science Foundation grant No 17-11-01261. 
The work of the second-named author (\S~4--5) was supported by RFBR grant No 18-31-20044.

The authors would like to express their gratitude to A.~Stavrova, A.~Stepanov and N.~Vavilov for their useful comments and interest in this work, and to the anonymous referee for his careful reading of an earlier version of the paper and for suggesting a number of improvements.

\section{Preliminaries}
If $x$ and $y$ are elements of a group $G$ then $[x, y]$ denotes the left-normed commutator $xyx^{-1}y^{-1}$.
We denote by $x^y$ (resp. ${}^y\!x$) the element $y^{-1}xy$ (resp. $yxy^{-1}$). 
Throughout the paper we make use of the following commutator identities:
\begin{align}
 [x, yz]^y               & =   [y^{-1}, x] \cdot [x, z], \label{rel43} \\ 
 \label{eq:H1ii} [xy, z] & =   {}^x[y, z] \cdot [x,z], \\
 \label{eq:H1ii-2}[x, yz]& =   [x, y] \cdot {}^{y}\![x, z], \\
 [[x, y], z] & =   {}^{x}\!\left({}^{y}[x^{-1}, [ y^{-1}, z]] \cdot {}^z[y, [ z^{-1}, x^{-1}]] \right). \label{HW-variant} \end{align}
Recall also from~\cite[Lemma~3.1.1]{RS76} that
\begin{equation} \label{eq:H1iii} [x,z] = 1 \text{ implies } [x, [y,z]] = [[x,y],{}^yz]. \end{equation} 

\subsection{Steinberg groups} \label{sec:Steinberg-intro}
Let $\Phi$ be a reduced and irreducible root system of rank $\ell \geq 2$ and $R$ be a commutative ring with $1$. Recall that in this case the \emph{Steinberg group} $\St(\Phi, R)$ can be defined by means of generators $x_{\alpha}(s)$ and relations:
\begin{align}
& \phantom{[}
x_\alpha(s) \cdot x_\alpha(t) = x_\alpha(s+t),\ \alpha\in\Phi,\ s,t\in R; \label{rel:add}\\
& [x_\alpha(s), x_\beta(t)] = \prod
 x_{i\alpha + j\beta}\left(N_{\alpha,\beta ij}\, s^i t^j\right),\quad \alpha,\beta\in\Phi,\ \alpha\neq-\beta,\ s,t\in R. \label{rel:CCF}
\end{align}
The indices $i$, $j$ appearing in the right-hand side of the above relation range over
all positive natural numbers such that $i\alpha + j\beta\in\Phi$.
The constants $N_{\alpha \beta i j}$ appearing in the right-hand side of \eqref{rel:CCF} are integers which can only take values $\pm 1,\pm 2,\pm 3$, they are called the {\it structure constants} of the Chevalley group $\GG(\Phi, R)$. Several different methods of computing signs of these constants have been proposed in the literature, see e.\,g.~\cite{V00}, \cite[\S~9]{VP}. 

For an additive subgroup $A\subseteq R$ and $\alpha \in \Phi$ we denote by $X_\alpha(A)$ the corresponding {\it root subgroup} of $\St(\Phi, R)$, i.\,e. the subgroup generated by all $x_\alpha(a)$, $a \in A$.

Whenever we speak of the Steinberg group $\St(\Psi, R)$ parameterized by a root subsystem $\Psi \subset \Phi$ we imply that the choice of structure constants 
 for $\St(\Psi, R)$ is compatible with that for $\St(\Phi, R)$ (i.\,e. the mapping $x_\alpha(\xi) \mapsto x_\alpha(\xi)$ yields a group homomorphism $\St(\Psi, R) \to \St(\Phi, R)$).

In this paper we will be mostly interested in the case when the Dynkin diagram of $\Phi$ is {\it simply-laced}, i.\,e. does not contain double bonds. In this case the defining relations of $\St(\Phi, R)$ have the following simpler form:
\begin{align}
x_{\alpha}(a)\cdot x_{\alpha}(b)&=x_{\alpha}(a+b), \tag{R1} \label{Steinberg-additivity}\\
[x_{\alpha}(a),\,x_{\beta}(b)]  &=x_{\alpha+\beta}(N_{\alpha\beta} \cdot ab),\text{ for }\alpha+\beta\in\Phi, \tag{R2} \label{Chevalley-CCF1} \\
[x_{\alpha}(a),\,x_{\beta}(b)]  &=1,\text{ for }\alpha+\beta\not\in\Phi\cup0. \tag{R3} \label{Chevalley-CCF2}
\end{align}
In the above formulae $a, b \in R$ and the integers $N_{\alpha, \beta} = N_{\alpha, \beta, 1, 1} = \pm 1$ are the structure constants of the Lie algebra of type $\Phi$. Although there is still some degree of freedom in their choice, they always must satisfy the relations, indicated in the following lemma (cf. \cite[\S~14]{VP}).
\begin{lemma} Suppose $\Phi$ is simply laced and $\alpha, \beta$ are roots of $\Phi$ such that $\alpha+\beta\in \Phi$, then the following identities hold for the structure constants:
\begin{equation} \label{eq:simplest} N_{\alpha, \beta} = -N_{\beta,\alpha} = - N_{-\alpha, -\beta} = N_{\beta, -\alpha-\beta} = N_{-\alpha-\beta, \alpha}. \end{equation}
If, moreover, $\gamma \in \Phi$ is such that $\alpha,\beta,\gamma$ form a basis of a root subsystem of type $\rA_3$ then one has
\begin{equation} \label{eq:cocycle} N_{\beta,\gamma} \cdot N_{\alpha, \beta+\gamma} = N_{\alpha+\beta, \gamma} \cdot N_{\alpha, \beta}. \end{equation} \end{lemma}
In our calculations below we repeatedly use identities~\eqref{eq:simplest} without explicit reference.

For $\alpha\in\Phi$ and $s \in R^\times$ we define certain elements $w_\alpha(s), h_\alpha(s)$ of $\St(\Phi, R)$ (the latter ones are sometimes called {\it semisimple root elements}):
\begin{align*} w_\alpha(s) & =  x_\alpha(s) \cdot x_{-\alpha}(-s^{-1}) \cdot x_\alpha(s), \\ h_\alpha(s) & =  w_\alpha(s) \cdot w_\alpha(-1).  \end{align*}

Recall from~\cite[Lemma~5.2]{Ma69} that the following relations hold for semisimple root elements:
\begin{align} \label{eq:conj-h-x} {}^{h_\alpha(t)}\!x_\beta(u) & = x_\beta(t^{\langle \beta,  \alpha \rangle}u), \\
              \label{eq:conj-h-h} {}^{h_\alpha(t)}\!h_\beta(u) & = h_\beta(t^{\langle \beta, \alpha \rangle} \cdot u) \cdot h_\beta(t^{\langle \beta,  \alpha \rangle})^{-1}, \\
              \label{eq:h-inv}    h_\alpha(t)^{-1}             & = h_{-\alpha}(t). \end{align}
In the above formulae $\langle \beta, \alpha \rangle$ denotes the integer $\tfrac{2(\beta, \alpha)}{(\alpha, \alpha)}$ (we denote by $(\text{-},\text{-})$ the inner product on the Euclidean space $\mathbb{R}^\ell$ containing $\Phi$).

\subsection{\texorpdfstring{$\K_2$}{K2}-groups and symbols} In our calculations we use two families of explicit elements of $\K_2(\Phi, R)$ called {\it Steinberg and Dennis--Stein symbols}. Notice that our notational conventions for symbols follow~\cite{DS73} and {\it not} more modern textbooks such as~\cite{Kbook} (cf. with~Definition~III.5.11 ibid.). Notice also that the operation on $\K_2(\Phi, R)$ is always written multiplicatively.

Recall that Steinberg symbols are defined for arbitrary $s, t \in R^\times$ as follows:
\begin{equation} \label{eq:steinberg} \{ s, t \}_\alpha = h_\alpha(st) \cdot h_\alpha^{-1}(s) \cdot h_\alpha^{-1}(t). \end{equation}
In turn, Dennis--Stein symbols are defined for arbitrary $a, b\in R$ satisfying $1 + ab \in R^\times$:
\begin{equation} \label{eq:dennis-stein}  \langle a,b \rangle _ \alpha = x_{-\alpha}\left(\tfrac{- b}{1 + ab}\right) \cdot x_{\alpha}(a) \cdot x_{-\alpha}(b) \cdot x_{\alpha}\left(\tfrac{- a}{1+ab}\right) \cdot h_{\alpha}^{-1}(1 + ab). \end{equation} 
Dennis--Stein symbol $\langle a, b \rangle_\alpha$ can be expressed through Steinberg symbols in the special case when either $a$ or $b$ is an invertible element of $R$. More specifically, the following formulae hold (cf.~\cite[p.~250]{DS73}).
\begin{equation} \label{DS-S-relationship} \langle a, b \rangle_\alpha = \{-a, 1+ab\}_\alpha\text{ for } a, 1+ab\in R^\times,\ \
 \{ s, t \}_\alpha = \left\langle -s, \tfrac{1 - t}{s} \right\rangle_\alpha\text{ for } s, t\in R^\times. \end{equation}
Steinberg and Dennis--Stein symbols depend only on the length of $\alpha$, in particular they do not depend on $\alpha$ if $\Phi$ is simply-laced. If $\Phi$ happens to be {\it nonsymplectic}, i.\,e. $\Phi \neq \rA_1, \rB_2, \rC_{\geq 3}$, Steinberg symbols are antisymmetric and bimultiplicative, i.\,e. they satisfy the following identities: \begin{equation} \label{eq:symbol-properties} \{ u, st \} = \{ u, s\} \{ u, t \}, \ \{ u, v \} = \{ v, u\}^{-1}. \end{equation}
For these and other properties of symbols we refer the reader to~\cite{DS73}.
From~\eqref{eq:symbol-properties} it follows that
\begin{equation} \label{eq:symbol-inverse} \{u, v\}^{-1} = \{u^{-1}, v\} = \{u, v^{-1}\}. \end{equation}

Recall that the classical Matsumoto theorem (see~\cite[Theorem~5.10]{Ma69}) allows one to compute the group $\K_2(\Phi, R)$ in the special case when $R=k$ is a field.
Using the modern language of Milnor--Witt K-theory (see~\cite{Mo04}) it can be formulated as follows:
\begin{equation*} \K_2(\Phi, k) = \left\{\begin{array}{ll} \K_2^\mathrm{MW}(k)& \text{if $\Phi$ is symplectic,}\\ \K_2^\mathrm{M}(k) & \text{otherwise.}\end{array}\right. \end{equation*}
In the following lemma we recall the computation of the group $\K_2(\Phi, R)$ in the case $R=k[X, X\inv]$.
\begin{lemma}[Hurrelbrink--Morita--Rehmann]\label{K2-laurent-field} Let $\Phi$ be a reduced irreducible root system and $k$ be an arbitrary field. Then there is a split exact sequence of abelian groups
\[\begin{tikzcd} 0 \arrow{r} & \K_2(\Phi, k) \arrow{r} & \K_2(\Phi, k[X, X^{-1}]) \arrow{r} & H(\Phi, k) \arrow{r} & 0, \end{tikzcd}\text{ in which}\]
The group $H(\Phi, k)$ admits the following description. One has $H(\Phi, k) \cong \K_1^\mathrm{M}(k) \cong k^\times$ if $\Phi$ is nonsymplectic. If $\Phi$ is symplectic and $\mathrm{char}(k)\neq 2$ then $H(\Phi, k) \cong \K_1^\mathrm{MW}(k)$. \end{lemma}
\begin{proof} Let us first consider the case of nonsymplectic $\Phi$, in which one can find a long root $\alpha\in \Phi$ in such a way that there is a commutative diagram of abelian groups
\[\begin{tikzcd}[column sep=5pc] k^\times \arrow{drr}[swap]{\{-, X\}} \arrow{r}{h}  & \K_2(\Phi, k[X, X^{-1}]) \arrow{r} & \K_2(\Phi, k(X)) \arrow{d}{\cong} \\
 & & \K_2^\mathrm{M}(k(X)), \end{tikzcd}\]
in which $h = \{ -, X \}_{\alpha}$ and the vertical map is an isomorphism by Matsumoto's theorem.
Notice that the diagonal map is split by the obvious residue homomorphism and therefore is injective. This, in turn, implies that $h$ is also injective.
For $\Phi \neq \rG_2$ the assertion of the lemma follows from~\cite[Satz~3]{Hur77} which asserts that for a nonsymplectic $\Phi\neq\rG_2$ one has $\K_2(\Phi, k[X, X\inv]) = \mathrm{Im}(h) \oplus \K_2(\Phi, k)$.
To obtain the assertion in the case $\Phi = \rG_2$ repeat the proof of~\cite[Satz~3]{Hur77} using \cite[Corollary~6]{AM88} instead of \cite[Korollar~4]{Hur77}.

Consider now the case when $\Phi$ is symplectic. In this case the assertion of the lemma is just a reformulation of~\cite[Theorem~B]{MR91},
which asserts that for $\ell \geq 1$ one has $\K_2(\rC_\ell, k[X, X\inv]) \cong \K_2(\rC_\ell, k) \oplus P(k)$, where
$P(k)$ is the set $k^\times \times I^2(k)$ with the group structure given by
\[ (u, y) \cdot (v, z) = (uv, y + z - \langle\langle u, v\rangle\rangle).\]
Here $I^2(k)$ stands for the second power of the fundamental ideal $I(k)$ in the Witt ring $W(k)$ of $k$.
Recall from~\cite{Mo04} that $\K_1^\mathrm{MW}(k)$ is isomorphic to the pullback of the diagram:
\[ \begin{tikzcd} \K_1^\mathrm{MW}(k) \arrow{r} \arrow{d} & I(k) \arrow{d} \\ \K_1(k) \arrow{r} & I(k)/I^2(k), \end{tikzcd} \]
in other words, it consists of pairs $[u, x]$ such that $x - \langle \langle u \rangle \rangle \in I^2(k)$.
It is easy to verify that the map $[u, x] \mapsto (u, \langle\langle u \rangle\rangle - x)$ defines an isomorphism of $\K_1^\mathrm{MW}(k)$ and $P(k)$. \end{proof}

\begin{rem} It is possible to give a different proof of the above lemma in the case $\Phi = \rC_\ell$ for $\ell \geq 5$ under the additional assumption that the characteristic of $k$ is not $2$. Indeed, we have the following chain of isomorphisms:
\begin{equation} \K_2(C_\infty, k[X, X^{-1}]) \cong \GW_2^{[2]}(k[X,X^{-1}]) \cong \GW_2^{[2]}(k) \oplus \GW_1^{[1]}(k) \cong \K_2^\mathrm{MW}(k) \oplus \K_1^\mathrm{MW}(k),\end{equation}
in which the first two isomorphisms are obtained from~\eqref{GW-concrete} and~\cref{bass-ft} and the last one follows from~\cite[Lemma~4.1.1]{AF17} and Matsumoto's theorem.
It remains to invoke the stability theorem~\cite[Theorem~9.4]{Pa89} to obtain the assertion in the unstable case. \end{rem}

\begin{lemma} \label{field-injectivity} The homomorphism $\St(\Phi, k[X]) \to \St(\Phi, k[X, X^{-1}])$ is injective. Moreover, the intersection of the images of $\St(\Phi, k[X])$ and $\St(\Phi, k[X\inv])$ inside $\St(\Phi, k[X, X\inv])$ coincides with the image of $\St(\Phi, k)$. \end{lemma}
\begin{proof} The first assertion follows from the commutative diagram
\[\begin{tikzcd}  & \K_2(\Phi, k) \arrow{dl}[swap]{\cong} \arrow[dr, hookrightarrow] & \\ \K_2(\Phi, k[X]) \arrow{rr} & & \K_2(\Phi, k[X, X^{-1}]) \end{tikzcd} \]
and the injectivity of its diagonal arrows. Indeed, the left arrow is an isomorphism by the Korollar of~\cite[Satz~1]{Re75} and the right arrow is split injective by~\cref{K2-laurent-field}.

Let us verify the second assertion. Let $g$ be an element of the intersection of $\St(\Phi, k[X])$ and $\St(\Phi, k[X\inv])$ inside $\St(\Phi, k[X, X\inv])$.
Clearly, the image of $g$ in $\GG(\Phi, k[X, X\inv])$ lies in $\GG(\Phi, k)$, therefore there exists $g_0 \in \St(\Phi, k)$ such that $gg_0^{-1} \in \K_2(\Phi, k[X]) = \K_2(\Phi, k)$. Thus, we conclude that $g \in \St(\Phi, k)$.
\end{proof}

\subsection{Relative Steinberg groups and unstable K-groups} \label{sec:quillen}
In this subsection we recall the definitions and basic facts pertaining to the theory of relative central extensions developed by J.-L.~Loday in~\cite{Lo78}. The main goal of this subsection is to show that Loday's theory can be applied to unstable Steinberg groups, and that the resulting relative unstable Steinberg groups have many of the properties of their stable counterparts. Some of the results of this subsection have been briefly mentioned in~\cite{S15} (cf. e.\,g. Corollaries 3--4).

Recall that the category of (commutative) pairs $\catname{Pairs}$ is defined as follows. Its objects are pairs $(R, I)$, in which $R$ is a commutative ring with $1$ and $I$ is a (not necessarily proper) ideal of $R$. A morphism of pairs $f \colon (R, I) \to (R', I')$ is, by definition, a (unit-preserving) ring homomorphism $f \colon R \to R'$ such that $f(I) \subseteq I'$. Notice that the mapping $(R, I) \mapsto (R \to R/I)$ defines a functor from $\catname{Pairs}$ to the morphism category $\catname{CRings}^\rightarrow$.
If $(R, I)$ is such that $R$ is a local ring with maximal ideal $I$, we call such pair a {\it local pair}.

There is an obvious fully faithful embedding $\catname{CRings} \to \catname{Pairs}$ sending $R$ to $(R, R)$. For a given functor $S \colon \catname{CRings} \to \catname{Groups}$ a {\it relativization} of $S$ is any functor $\widetilde{S} \colon \catname{Pairs} \to \catname{Groups}$ extending $S$ in the obvious sense. Relativization of a functor is not unique.

Recall that {\it the double ring} $D_{R, I}$ of a pair $(R, I)$ is, by definition, the pullback ring $R \times_{R/I} R$. In other words, it is the ring consisting of pairs of elements of~$R$ congruent modulo $I$. Denote by $p_0, p_1, \Delta$ the two obvious projections and the diagonal map 
\begin{tikzcd} D_{R, I} \arrow[r, shift right=2] \arrow[r] & R. \arrow[l, shift right=2] \end{tikzcd} It is clear that $p_0 \Delta = p_1 \Delta = id_{R}$.

Let $S \colon \catname{CRings} \to \catname{Groups}$ be a functor. Set $G_i = \Ker(S(p_i))$ and define {\it Loday's relativization} $S(R, I)$ as $ G_0 / [G_0, G_1]$. The homomorphism $S(p_1)$ induces a natural transformation $S(R, I) \to S(R)$. We denote this map by $\mu = \mu_{R,I}$ and its kernel by $C_S(R, I)$: 
\begin{equation} \label{LodayRelativization} \begin{tikzcd} 1 \arrow{r} & C_S(R, I) \arrow{r} & S(R, I) \arrow{r}{\mu} & S(R) \arrow{r} & S(R/I) \arrow{r} & 1. \end{tikzcd} \end{equation}

\begin{df} By definition, the {\it relative Steinberg group} $\St(\Phi, R, I)$ is the result of application of Loday's relativization to the functor $\St(\Phi, -)$. Notice that $\St(\Phi, R, I)$ is a central extension of $\Ker(\St(\Phi, R) \to \St(\Phi, R/I))$ by the abelian group $C_{\St(\Phi, -)}(R, I)$. For shortness we rename the latter group to $C(\Phi, R, I)$. \end{df}

Our next goal is to obtain a homological interpretation of the group $C(\Phi, R, I)$.
In order to do this, we need to recall some additional notation and terminology.

First of all, recall that a {\it central extension} of a group $G$ is a surjective homomorphism $\widetilde{G} \to G$, whose kernel is contained in the center of $\widetilde{G}$. 
A morphism of central extensions is a homomorphism $\widetilde{G} \to \widetilde{G}'$ over $G$.
A central extension is said to be {\it universal} if it is an initial object of the category of central extensions of $G$.

Recall that a {\it crossed module} is a triple $(M, N, \mu)$ consisting of the following data:
\begin{enumerate} [label=\normalfont(\arabic*)]
 \item a group $N$ acting on itself by left conjugation (i.\,e. ${}^{n}\!n' = n n' n^{-1}$).
 \item a group $M$ with a left action of $N$ (we call such a group an $N$-group and use the notation ${}^n m$ to denote the image of an element $m\in M$ under the action of $n \in N$);
 \item a homomorphism $\mu \colon M\to N$ preserving the action of $N$ and satisfying {\it Peiffer identity} ${}^{\mu(m)}\!m' = m m' m^{-1}$.
\end{enumerate}
It can be shown that the image of $\mu$ is always a normal subgroup of $N$ and that the kernel of $\mu$, which we denote by $L$, is always contained in the center of $M$.
 
Let $\nu \colon N \twoheadrightarrow Q$ be a surjective homomorphism.
A {\it relative central extension of $\nu$} is, by definition,
a crossed module $(M, N, \mu)$ such that the cokernel of $\mu$ is $\nu$:
\begin{equation} \label{RelativeCentralExtension} \begin{tikzcd} 1 \arrow{r} & L \arrow{r} & M \arrow{r}{\mu} & N \arrow{r}{\nu} & Q \arrow{r} & 1 \end{tikzcd} \end{equation}

A morphism $(M, \mu) \to (M', \mu')$ of two relative central extensions of $\nu$ is, by definition, an $N$-group homomorphism $f\colon M \to M'$ such that $\mu' f = \mu$. 
A relative central extension is said to be {\it universal} if it is an initial object of the category of relative central extensions of $\nu$. 

It turns out that the set $Ext(Q, N; L)$ of isomorphism classes of relative central extensions of $\nu$ by an abelian group $L$ can be classified by means of a certain cohomological invariant called {\it characteristic class}. More precisely, \cite[Th{\'e}or{\`e}me~1]{Lo78} asserts that there is a well-defined bijection $\xi \colon Ext(Q, N; L) \to \HH^3(Q, N; L)$.
 
For the rest of this subsection $S \xrightarrow{\pi} P \subseteq G$ is a triple of group-valued functors on the category of commutative rings satisfying the following assumptions:
\begin{enumerate} [label=\normalfont(A\arabic*)]
 \item \label{req:left-exact} $G(D_{R, I}) \cong G(R) \times_{G(R/I)} G(R)$.
 \item \label{req:coeq} For every pair $(R, I)$ the coequalizer of $S(p_0), S(p_1)$ is precisely $S(R) \to S(R/I)$.
 \item \label{req:subfunc} $P(R)$ is a normal subgroup of $G(R)$.
 \item \label{req:uce} The homomorphism $ \pi_R \colon S(R) \to P(R)$ is a universal central extension for all $R$. In particular, $S(R)$ and $P(R)$ are perfect and $\HH_2(S(R), \ZZ) = 0$.
\end{enumerate}

\begin{lemma}\label{lem:relativeH3}
 For every pair $(R, I)$ the homomorphism $\mu \colon S(R, I) \to S(R)$ is a universal relative central extension of $\nu \colon S(R) \to S(R/I)$. The group $C_S(R, I)$ is naturally isomorphic to the relative homology group $\HH_3(S(R/I), S(R); \ZZ)$.
\end{lemma}
\begin{proof}
The action of $S(R)$ on $S(D_{R, I})$ given by ${}^g h = S(\Delta)(g) \cdot h \cdot S(\Delta)(g)^{-1}$ induces an action of $S(R)$ on $S(R, I)$.
The homomorphism $\mu \colon S(R, I) \to S(R)$ from~\eqref{LodayRelativization} is an $S(R)$-homomorphism with respect to this action.
From~\ref{req:left-exact} and $\pi(G_i) \subseteq \Ker(G(p_i))$ we obtain that 
$G_0 \cap G_1 \subseteq \Ker(\pi_{D_{R, I}})$ hence it is a central subgroup of $S(D_{R,I})$ by~\ref{req:uce}. Thus, we have verified the assumptions of~\cite[Proposition~6]{Lo78} which asserts that the homomorphism $\mu$ is a universal relative central extension of the coequalizer $\nu = \mathrm{coeq}(d_0, d_1)$. Since $\nu$ coincides with $S(R) \to S(R/I)$ by~\ref{req:coeq}, we have completed the proof of the first assertion of the lemma.

Set $N = S(R)$, $Q = S(R/I)$, $C = \HH_3(Q, N; \ZZ)$. Recall from the proof of~\cite[Th{\'e}or{\`e}me~2]{Lo78} that to every relative central extension $(M, \mu)$ of $\nu$ with kernel $L$ one can associate a homomorphism of abelian groups $C \to L$. This homomorphism is obtained from the characteristic class $\xi(M, \mu)$ via the isomorphism $\HH^3(Q, N; L) \cong \mathrm{Hom}(C, L)$ of the universal coefficients theorem for cohomology (for this isomorphism we need the vanishing of $\HH_2(Q, N; \ZZ)$, which is a consequence of~\ref{req:uce}).

In the special case $M = S(R, I)$ this construction produces a homomorphism $C \to C_{S}(R, I)$ whose naturality in $(R, I)$ follows from~\cite[Proposition~3]{Lo78}. 
This homomorphism is an isomorphism by~\cite[Th{\'e}or{\`e}me~2]{Lo78}. \end{proof}

We retain our notation for the functors $S, P$ and $G$.
For $i\geq 1$ we define the unstable Quillen K-functors $\K_{i}^{G, P}$ via
\begin{equation} \label{plus-constr} \K_i^{G,P}(R) = \pi_i(BG(R)^+_{P(R)}). \end{equation}
It is not hard to obtain the following concrete description of these functors in the cases $i=1,2,3$.
\begin{lemma}\label{lem:lowerKgroups} There are natural isomorphisms \begin{enumerate} [label=\normalfont(\arabic*)]
 \item $\K_1^{G,P}(R) \cong G(R) / P(R)$;
 \item $\K_2^{G,P}(R) \cong \Ker(S(R) \to G(R))$;
 \item $\K_3^{G,P}(R) \cong \HH_3(S(R), \ZZ).$ \end{enumerate} \end{lemma}
\begin{proof} The first claim is obvious, the other claims follow from~\ref{req:subfunc} and~\ref{req:uce} using the standard properties of the plus-construction, see \cite[\S~IV.1]{Kbook}   (cf. Exercises~1.8--1.9 ibid.) \end{proof}

We denote by $\mathrm{O}_{2n}(R)$ and $\mathrm{EO}_{2n}(R)$ the orthogonal group of rank $n$ over a ring $R$ and its elementary subgroup, respectively (see e.\,g.~\cite{Su82} for the definition of these groups).
Now set $G_n = \mathrm{O}_{2n}(-)$, $P_n = \mathrm{EO}_{2n}(-)$, $S_n = \St(\rD_\ell, -)$.
The triple $(G_n, P_n, S_n)$ plays a key role in the present paper.

We need to introduce more notation. Denote by $\widetilde{G}_n$ the Chevalley group $\GG(\rD_n, -) = \mathrm{Spin}(2n, -)$
 and by $\widetilde{P}_n$ its elementary subfunctor $\mathrm{E}(\rD_n, -) = \mathrm{Epin}(2n, -)$.

It is clear that both $G_n$ and $\widetilde{G}_n$ satisfy~\ref{req:left-exact} and that $S_n$ satisfies \ref{req:coeq}.
It is well known that $P_n$ (resp. $\widetilde{P}_n$) is a normal subgroup of $G_n$ (resp. $\widetilde{G}_n$) for $n\geq 3$ (see~\cite{Su82},\cite{Ta86}). Thus,~\ref{req:subfunc} is satisfied for $P_n$ and $\widetilde{P}_n$.

By~\cite[Corollary~5.4]{St71} the group $S_n$ is {\it centrally closed} for $n \geq 5$, in other words every central extension of $S_n$ splits.
On the other hand, \cite[Theorem~1]{LS17} shows that $\K_2(\rD_\ell, R) = \Ker(S_n \to \widetilde{G}_n)$ is a central subgroup of $S_n$ for $n\geq 4$. 

Recall from p.~189 of~\cite{Ba74} that for $n\geq 3$ there is an exact sequence
\begin{equation}\label{eq:Bass-ES} \begin{tikzcd} 1 \arrow{r} & \mu_2(R) \arrow{r} & \widetilde{P}_n(R) \arrow{r} & P_n(R) \arrow{r} & 1, \end{tikzcd} \end{equation}
in which the group $\mu_2(R) = \{ a \in R^\times \mid a^2 = 1 \}$ is a central subgroup of $\widetilde{P}_n(R)$.
Now by~\cite[\S~7(v)]{St67} the group $\KO_2(2n, R) = \Ker(S_n \to G_n)$ is also central in $S_n$.
Thus, we see that $S_n$ is a universal central extension of both $P_n$ and $\widetilde{P}_n$ for $n\geq 5$.

We have checked that both triples $(G_n, P_n, S_n)$ and $(\widetilde{G}_n, \widetilde{P}_n, S_n)$ satisfy the requirements~\ref{req:left-exact}--\ref{req:uce} for $n \geq 5$. In the sequel we use the notation $\KO_i(2n, R)$ as a shorthand for $\K_i^{G_n, P_n}(R)$. 

Notice that the groups $\KO_2(2n, R)$ and $\K_2(\rD_n, R)$ are related via the exact sequence of~\cite[Theorem~7.2.12]{HOM}:
\[ \begin{tikzcd} 1 \arrow{r} & \K_2(\rD_n, R) \arrow{r} & \KO_2(2n, R) \arrow{r} & \mu_2(R) \arrow{r} & 1. \end{tikzcd} \]

We conclude this subsection with the following stability result (see~\cite[Theorem~9.4]{Pa89}).
\begin{externaltheorem}[Panin] \label{Panin-stability}
 Let $R$ be either a field, principal ideal domain or a Dedekind domain. Set $a = 1,2$ or $3$ in each of these three cases, respectively.
 Then the stability map $\KO_i(2n, R) \to \KO_i(2(n+1), R)$ is an epimorphism for $n \geq b$ 
 and an isomorphism for $n \geq b + 1$, where $b = \mathrm{max}(2i, a+i-1)$. \end{externaltheorem}

\section{An injectivity theorem for Steinberg groups} \label{firstPart}
We start this section by recalling basic notation and facts pertaining to the theory of higher Grothendieck--Witt groups. Recall that this theory, developed by M.~Schlichting, is a modern broad generalization of the classical hermitian K-theory of rings. We refer the reader to~\cite[\S~2]{FRS12} and~\cite[\S~2]{AF17} for an introduction to Grothendieck--Witt groups.

For our purposes it suffices to restrict attention to the affine case, in which  the Grothendieck--Witt groups $\GW_i^{[k]}(R)$ for $i \geq 1,\ [k] \in \ZZ/4\ZZ$ can be considered simply as a shorthand for the following 4 groups:
\begin{equation} \label{GW-concrete} \GW_i^{[k]}(R) = \left\{\begin{array}{ll} \KO_i(R), & k = 0 \\ {}_{-1}\!U_i(R), & k = 1 \\ \mathrm{KSp}_i(R), & k = 2 \\ U_i(R), & k = 3. \end{array}\right. \end{equation}
Here $\mathrm{KO}_i(R)$ denotes the usual orthogonal K-group defined via~\eqref{plus-constr} with $G(R) = O_\infty(R)$ and $P(R) = [G(R), G(R)]$.
Replacing the stable orthogonal group with the stable symplectic group one can also define the symplectic K-groups $\mathrm{KSp}_i(R)$.
We refer the reader to~\cite{Ka80} for the definition and properties of the groups ${}_{\pm 1}\!U_i(R)$. We will not use these definitions directly.

The following result, which is a special case of~\cite[Theorem~9.13]{Sch16} of M.~Schlichting, plays a key role in the proof of~\cref{thm41}.
\begin{externaltheorem}[Bass Fundamental Theorem]\label{bass-ft} Suppose that $R$ is a regular ring such that $2 \in R^\times$, 
then for any $i\geq 1$, $k\in \ZZ/4\ZZ$ there is a natural split exact sequence of abelian groups 
\[ \begin{tikzcd} 0 \arrow{r} & \GW_i^{[k]}(R) \arrow{r} & \GW_i^{[k]}(R[X, X^{-1}]) \arrow{r} & \GW_{i-1}^{[k-1]}(R) \arrow{r} & 0. \end{tikzcd} \] \end{externaltheorem}
We will need only the special case $k=0$ of the above theorem, in which it turns into an earlier result of J.~Hornbostel, see~\cite[Corollary~5.3]{Ho05}.

For the rest of this section let us fix the following notation.
Let $A$ be an arbitrary commutative local ring with maximal ideal $M$ and residue field $k$.
Denote by $B$ the subring $A[X^{-1}] + M[X]$ of the ring $R = A[X, X^{-1}]$ and
by $I$ the ideal $M[X, X^{-1}]$ of $R$ (it is clear that $I$ is also an ideal of $B$).

\begin{lemma} \label{lem:prop41}
Assume additionally that the residue field $k$ is of characteristic $\neq 2$.
Then the canonical homomorphism $f\colon C(\rD_\ell, B, I) \to C(\rD_\ell, R, I)$ is surjective for $\ell \geq 7$. \end{lemma}
\begin{proof}
Writing down the starting portions of the homology long exact sequences for the homomorphisms $\St(\rD_\ell, R) \to \St(\rD_\ell, R/I)$, $\St(\rD_\ell, B) \to \St(\rD_\ell, B/I)$ and using the isomorphisms of~\cref{lem:relativeH3} and~\cref{lem:lowerKgroups} we obtain the following commutative diagram:
\[\begin{tikzcd} \KO_3(2\ell, B) \arrow{r} \arrow{d} & \KO_3(2\ell, k[X^{-1}]) \arrow{d}[swap]{f'} \ar[r, twoheadrightarrow] & \ar{d}[swap]{f} C(\rD_\ell, B, I) \\
 \KO_3(2\ell, R) \arrow{r} & \KO_3(2\ell, k[X, X^{-1}]) \ar[r, twoheadrightarrow] & C(\rD_\ell, R, I). \end{tikzcd} \]
By~\cref{Panin-stability} the homomorphism $f'$ can be identified with the canonical map $\GW_3^{[0]}(k[X^{-1}]) \to \GW_3^{[0]}(k[X, X^{-1}])$.
By~\cref{bass-ft} $\GW_3^{[0]}(k[X, X^{-1}]) \cong \GW_3^{[0]}(k) \oplus \GW_2^{[3]}(k),$
but since the group $\GW_2^{[3]}(k)$ is trivial by~\cite[Lemma~2.2]{FRS12}, the homomorphism $f'$ (and hence $f$) is surjective.
\end{proof} 
 
We will also need the following property of relative Steinberg groups which is a special case of a more general property discussed in~\cite[\S~2]{LS17}.
\begin{lemma}\label{lem:lemma32} Let $\Phi$ be a simply-laced root system of rank $\geq 3$,
Consider the following commutative square of canonical homomorphisms.
\[\begin{tikzcd} \St(\Phi, B, I) \arrow{r}{\mu_B} \arrow{d} & \St(\Phi, B) \arrow{d} \\ \St(\Phi, R, I) \arrow{r}{\mu_R} \ar[ur, "t", dashrightarrow] & \St(\Phi, R) \end{tikzcd} \]
Then there exists a diagonal homomorphism $t$ which makes the diagram commute.   
\end{lemma} 
\begin{proof}
 Notice that $R$ is isomorphic to the principal localisation of $B$ at $X$
  and that $I$ is uniquely $X$-divisible in the sense of~\cite[\S~4]{LS17}.
 Thus, in the special case $\Phi = \rA_3$ the assertion of the lemma follows from~\cite[Theorem~3]{LS17}.
 In the general case the assertion of the lemma is a corollary of amalgamation theorem~\cite[Theorem~9]{S15}.
\end{proof}

\begin{theorem} \label{thm41} Suppose that $2 \in A^\times$. Then for $\ell \geq 7$ the canonical homomorphism $\St(\rD_\ell, B) \to \St(\rD_\ell, R)$ is injective. \end{theorem}
\begin{proof}
 Consider the following commutative diagram with exact rows obtained from~\eqref{LodayRelativization}:
\[\begin{tikzcd} C(\rD_\ell, B, I) \arrow{r}{\lambda_B} \arrow[d, swap, "f", twoheadrightarrow] & \St(\rD_\ell, B, I) \arrow{r}{\mu_B} \arrow{d}[swap]{g} & \St(\rD_\ell, B) \arrow{r}{\nu_B} \arrow{d}[swap]{h} & \St(\rD_\ell, k[X^{-1}]) \arrow{d}[swap]{i} \\ C(\rD_\ell, R, I) \arrow{r}{\lambda_R} & \St(\rD_\ell, R, I) \arrow{r}{\mu_R} \arrow[ur, "t", dashrightarrow] & \St(\rD_\ell, R) \arrow{r}{\nu_R} & \St(\rD_\ell, k[X, X^{-1}]).\end{tikzcd}\]
The lifting $t$ in the above diagram is obtained from~\cref{lem:lemma32}.
Let $a$ be an element of $\Ker(h)$. Since $i$ is injective by~\cref{field-injectivity}, the element $a$
 also lies in $\Ker(\nu_B)$ and hence comes from some $b \in \St(\rD_\ell, B, I)$ via $\mu_B$.
Since $g(b) \in \Ker(\mu_R)$ there exists some $c \in C(\rD_\ell, R, I)$ such that $\lambda_R(c) = g(b)$. 
By~\cref{lem:prop41}, $f$ is surjective, therefore $c = f(d)$ for some $d \in C(\rD_\ell, B, I)$.
The required assertion now follows from the following calculation:
 \[ 1 = \mu_B\lambda_B(d) = tg\lambda_B(d) = t\lambda_Rf(d) =t(g(b)) = \mu_B(b) = a. \qedhere \]
\end{proof}

\section{Elementary calculations in relative Steinberg groups}\label{sec:elementary}
Throughout this section $\Phi$ denotes an irreducible root system of rank $\geq 2$, $R$ a commutative ring, and $I, J$ denote a pair of ideals of $R$.
Unless stated otherwise we assume $\Phi$ to be simply laced.

We denote by $\overline{\St}(\Phi, R, I)$ the kernel of the homomorphism $\St(\Phi, R) \to \St(\Phi, R/I)$.
This group coincides with the image in $\St(\Phi, R)$ of the relative group $\St(\Phi, R, I)$ defined in~\cref{sec:quillen}.

\subsection{Generators of relative Steinberg groups}
Denote by $\St(\Phi, I)$ the subgroup of $\St(\Phi, R)$ generated as a group by all root subgroups $X_\alpha(I)$, $\alpha\in\Phi$.
It is clear that $\overline{\St}(\Phi, R, I)$ contains $\St(\Phi, I)$ and, in fact, is its normal closure.
We also denote by $\myol{H}(\Phi, R, I)$ the subgroup of $\overline{\St}(\Phi, R, I)$ generated by the semisimple root elements $h_\alpha(u)$ and symbols $\{u, v\}$, $u \in (1+I)^\times$, $v \in R^\times$, $\alpha\in \Phi$.

For $s\in I$, $\xi \in R$, $t \in J$ we define the following two families of elements of $\overline{\St}(\Phi, R, I)$ 
\begin{align}
 z_\alpha(s, \xi) &= x_\alpha(s)^{x_{-\alpha}(\xi)}, \label{Zdef}\\
 c_\alpha(s, t)   &= [x_\alpha(s), x_{-\alpha}(t)]. \label{Cdef}
\end{align}

\begin{lemma}\label{Zrels} Let $\Phi$ be a simply laced root system.
The elements $z_\alpha(s, \xi)$ satisfy the following relations for all $\xi, \eta\in R$, $s\in I$:
\begin{enumerate} [label=\normalfont(\arabic*)]
\item\label{Z1} $z_{\alpha}(s, \xi) ^ {x_{-\alpha}(\eta)} = z_{\alpha}(s, \xi + \eta)$;
\item\label{Z2} $z_{\beta}(s, \xi) ^ {x_{\alpha}(\eta)} = x_{\alpha} (- s\xi \eta) \cdot x_{\alpha+\beta} (N_{\beta, \alpha}\cdot s\eta)     \cdot z_{\beta}(s, \xi)\ \text{if}\ \alpha + \beta \in \Phi$;
\item\label{Z3} $z_{\beta}(s, \xi) ^ {x_{\alpha}(\eta)} = x_{\alpha} (s\xi \eta) \cdot x_{\alpha-\beta} (N_{\beta,-\alpha}\cdot s\xi^2\eta) \cdot z_{\beta}(s, \xi)\ \text{if}\ \alpha - \beta \in \Phi$;

\item\label{Z4} $z_{\beta}(s, \xi) ^ {x_{\alpha}(\eta)} = z_{\beta}(s, \xi)\ \text{if}\ \alpha\perp\beta$;
\item\label{Z5} If $\alpha+\beta\in\Phi$ then
\begin{multline*} z_{\alpha+\beta}(s\eta, \xi) = x_\alpha(\epsilon s)\cdot x_{-\beta}(-s\xi) \cdot x_{\beta}(s\xi\eta^2) \cdot x_{\alpha+\beta}(s \eta) \cdot \\ \cdot z_\alpha(-\epsilon s, -\epsilon \xi\eta) \cdot
  x_{-\alpha}(-\epsilon s\xi^2\eta^2) \cdot x_{-\alpha-\beta}(- s \xi^2 \eta) \cdot z_{-\beta}(s\xi, -\eta),\end{multline*}
 where $\epsilon = N_{\alpha,\beta}$.
\end{enumerate} \end{lemma}
\begin{proof}
The first four assertions are contained in~\cite[Lemma~9]{S15}, so it remains to verify the last assertion.
Notice that by~\eqref{Chevalley-CCF2} the elements $x_\beta(\eta)$ and $x_{-\alpha}(\epsilon s \eta)$ commute with each other, therefore from~\eqref{Chevalley-CCF1}--\eqref{Chevalley-CCF2} we obtain that
\begin{multline} \nonumber
  z_{\alpha+\beta}(s\eta, \xi) = [x_\alpha(\epsilon s)^{x_{-\alpha-\beta}(\xi)}, x_\beta(\eta)^{x_{-\alpha-\beta}(\xi)}] = \\
  = [x_\alpha(\epsilon s) x_{-\beta}(-s\xi), x_{\beta}(\eta) x_{-\alpha}(\epsilon \xi\eta)] = \\ 
  = x_\alpha(\epsilon s) \cdot x_{-\beta}(-s\xi) \cdot z_\alpha(-\epsilon s, -\epsilon \xi\eta)^{x_{\beta}(-\eta)} \cdot z_{-\beta}(s\xi, -\eta)^{x_{-\alpha}(-\epsilon \xi\eta)},
\end{multline} 
and the required assertion follows from~\ref{Z2}.
\end{proof}

Let us mention an immediate application of~\cref{Zrels}.
First of all, recall the following two results which give two different generating sets for the group $\overline{\St}(\Phi, R, I)$ (notice that both results apply to general irreducible $\Phi$ of rank $\geq 2$, i.\,e. not necessarily simply-laced ones).
\begin{externaltheorem}[Stein--Tits--Vaserstein] \label{thm:Tits} The group $\overline{\St}(\Phi, R, I)$ is generated (as an abstract group) by elements $z_\alpha(s, \xi)$, $\alpha \in \Phi$, $s \in I$, $\xi \in R$. \end{externaltheorem} \begin{proof} See e.\,g.~\cite[Theorem 2]{Va86}. \end{proof}

\begin{df} \label{df:root-subsets}
Recall from~\cite[Ch.~VI, \S~1.7]{Bou81} that a root subset $S \subseteq \Phi$ is called {\it closed} if $\alpha, \beta \in S$ and $\alpha+\beta\in \Phi$ imply $\alpha + \beta \in S$. 
Recall that a closed root subset $S$ is called {\it parabolic} (resp. {\it symmetric}, resp. {\it special}) if $S \cup -S = \Phi$ (resp. $S = -S$, resp. $S \cap (-S) = \varnothing$).
The {\it special part} $\Sigma_S$ of a parabolic subset $S$, by definition, consists of all $\alpha \in S$ such that $-\alpha \not\in S$.
\end{df}

For a subset of roots $U \subseteq \Phi$ we denote by $\mathcal{Z}(U, R, I)$ the set consisting of elements $x_\alpha(s)$, $s \in I$, $\alpha \in \Phi$ and $z_\alpha(s, \xi)$, $\alpha \in U$, $s\in I$, $\xi \in R$.
\begin{externaltheorem}[Stepanov] \label{thm:Stepanov} 
Let $S \subseteq \Phi$ be a parabolic subset of $\Phi$. Then the group $\overline{\St}(\Phi, R, I)$ is generated by the set $\mathcal{Z}(\Sigma_S, R, I)$.
 \end{externaltheorem} \begin{proof} See~\cite[Lemma~4]{S15}. \end{proof}

\begin{rem} We claim that in the simply-laced case the stronger \cref{thm:Stepanov} can be deduced from~\cref{thm:Tits} by means of~\cref{Zrels}. Indeed, consider the operator $d \colon 2^\Phi \to 2^\Phi$ of root subsets, whose value on each $U \subseteq \Phi$ is given by \[d(U) = U \cup \left((U - U)\cap \Phi\right).\]
Here $(U-U) \cap \Phi$ denotes the set of all differences of roots from $U$ which are themselves roots. It is not hard to show that for any parabolic subset $S \subseteq \Phi$ the subset $\Sigma_S$ has the property that $d^n(\Sigma_S) = \Phi$ for some $n>1$ (in fact, $n=2$). It remains to see that \cref{Zrels}\ref{Z5} immediately implies that every group $G$ containing $\mathcal{Z}(U, R, I)$ also contains $\mathcal{Z}(dU, R, I)$. \end{rem}

\begin{lemma} \label{Crels} Let $\Phi$ be a simply-laced root system.
The elements $c_\alpha(s, t)$ satisfy the following relations for all $s\in I,\ t\in J,\ \xi\in R$:
 \begin{enumerate} [label=\normalfont(\arabic*)]
 \item \label{C1} $[c_\beta(s, t),\ x_{\alpha}(\xi)] = x_{\alpha}(- st\xi) \cdot x_{\alpha+\beta}(N_{\alpha,\beta}\cdot s^2t\xi)$ if $\alpha+\beta \in \Phi$;
 \item \label{C2} $[c_\beta(s, t),\ x_{\alpha}(\xi)] = x_{\alpha}(st\xi + s^2t^2\xi) \cdot x_{\alpha-\beta}(N_{-\alpha, \beta}\cdot st^2\xi)$ if $\alpha-\beta \in \Phi$;  
 \item \label{C3} $[c_\beta(s, t),\ x_{\alpha}(\xi)] = 1$ if $\alpha \perp \beta$;  
 \item \label{C4} If $\alpha+\beta\in\Phi$ then
  \begin{equation*} c_{\alpha+\beta}(s, t\xi) = [x_{\beta}(st),\ x_{-\beta}(\xi)] ^ {x_{\alpha+\beta}(-s) x_{-\alpha}(\epsilon t)} \cdot c_{\alpha}(\epsilon s\xi, -\epsilon t)^{-1} \cdot x_{-\beta}(-st\xi^2),\end{equation*}
  where $\epsilon = N_{\alpha,\beta}$.
 \end{enumerate}
\end{lemma}
\begin{proof}
Notice that~\ref{C3} is an obvious consequence of~\eqref{Chevalley-CCF2}. Let us verify the first two assertions. Suppose that $\alpha + \beta \in \Phi$.  We have that
\begin{align*} \nonumber [c_{\beta}(s, t),\ x_{\alpha}(\xi)] = [x_{\beta}(s),\ x_{-\beta}(t)] \cdot [x_{-\beta}(t),\ x_{\beta}(s)\cdot x_{\alpha + \beta}(N_{\alpha, \beta} \cdot s\xi)] &\text{ by~\eqref{Cdef}, \eqref{Chevalley-CCF1}, \eqref{Chevalley-CCF2}} \\
 = {}^{x_\beta(s)}\![x_{-\beta}(t),\ x_{\alpha+\beta}(N_{\alpha, \beta} \cdot s\xi)] &\text{ by~\eqref{eq:H1ii-2}} \\ = x_{\alpha}(- st\xi) \cdot x_{\alpha+\beta}(N_{\alpha,\beta} \cdot s^2t\xi) &\text{ by~\eqref{Chevalley-CCF1}.} \end{align*}
 Now suppose that $\alpha - \beta \in \Phi$. We have that
\begin{align*} \nonumber
\begin{split}[[x_\beta(s),\ x_{-\beta}(t)],\ x_{\alpha}(\xi)] = {}^{x_\beta(s)}\!\bigl( {}^{x_{-\beta}(t)}[x_\beta(-s),\ [x_{-\beta}(-t),\ x_{\alpha}(\xi)]] \cdot \hspace{50pt} \\ \cdot {}^{x_{\alpha}(\xi)}[x_{-\beta}(t),\ [ x_{\alpha}(-\xi),\ x_\beta(-s)]] \bigr)^{} & \text{ by~\eqref{HW-variant}}\end{split} \\
= {}^{x_\beta(s) \cdot x_{-\beta}(t)}x_{\alpha}\left(-N_{\beta, \alpha-\beta} \cdot N_{\alpha,-\beta} \cdot s t \xi\right) = {}^{ x_\beta(s) \cdot x_{-\beta}(t)}x_{\alpha}(s t \xi) & \text{ by~\eqref{Chevalley-CCF1},\eqref{Chevalley-CCF2}}\\
= x_{\alpha}(st\xi) \cdot x_{\alpha-\beta}(-N_{\alpha,-\beta}\cdot st^2\xi)^{x_\beta(-s)} = x_{\alpha}(st\xi + s^2t^2 \xi) \cdot x_{\alpha-\beta}(N_{-\alpha,\beta}\cdot st^2\xi) & \text{ by~\eqref{Chevalley-CCF1}.} \end{align*}
Finally, let us prove~\ref{C4}. Suppose that $\alpha + \beta \in \Phi$. The required identity can be obtained as follows:
\begin{align*}
[x_{\alpha+\beta}(s),\ x_{-\alpha-\beta}(t\xi)] = [x_{\alpha+\beta}(s),\ [x_{-\alpha}(-\epsilon t),\ x_{-\beta}(\xi)]]&\ \text{by~\eqref{Chevalley-CCF1}}\\
\begin{split}
    ={}^{x_{-\alpha}(-\epsilon t)}\!\bigl( {}^{x_{\alpha+\beta}(s)}[[x_{\alpha+\beta}(-s),\ x_{-\alpha}(\epsilon t)],\ x_{-\beta}(\xi)]  \cdot \hspace{70pt} \\
        \cdot {}^{x_{-\beta}(\xi)}[[ x_{-\beta}(-\xi),\ x_{\alpha+\beta}(s)],\ x_{-\alpha}(\epsilon t)]\bigr) &\ \text{by~\eqref{HW-variant}}
\end{split}\\
= \left({}^{x_{\alpha+\beta}(s)}[x_{\beta}(st),\ x_{-\beta}(\xi)] \cdot {}^{x_{-\beta}(\xi)}[x_{\alpha}(\epsilon s\xi),\ x_{-\alpha}(\epsilon t)]\right)^{x_{-\alpha}(\epsilon t)} &\ \text{by~\eqref{Chevalley-CCF1}}\\
= \left({}^ {x_{\alpha+\beta}(s)}[x_{\beta}(st),\ x_{-\beta}(\xi)] \cdot [x_{\alpha}(\epsilon s\xi),\ x_{-\alpha}(\epsilon t) \cdot x_{-\alpha-\beta}(t\xi)]\right)^{x_{-\alpha}(\epsilon t)} &\ \text{by~\eqref{Chevalley-CCF1},\eqref{Chevalley-CCF2}}\\
= [x_{\beta}(st),\ x_{-\beta}(\xi)] ^ {x_{\alpha+\beta}(-s) x_{-\alpha}(\epsilon t)} \cdot [x_{-\alpha}(-\epsilon t),\ x_{\alpha}(\epsilon s\xi)] \cdot [x_{\alpha}(\epsilon s\xi),\ x_{-\alpha-\beta}(t\xi)] &\ \text{by~\eqref{rel43}}\\
= [x_{\beta}(st),\ x_{-\beta}(\xi)] ^ {x_{\alpha+\beta}(-s) x_{-\alpha}(\epsilon t)} \cdot c_{\alpha}(\epsilon s\xi, -\epsilon t)^{-1} \cdot x_{-\beta}(-st\xi^2)  &\ \text{by~\eqref{Chevalley-CCF1}}. \qedhere \end{align*}
\end{proof}

\subsection{Computation of the kernel of the map of evaluation at 0}\label{sec:computationOfK}
Let $A$ be a local ring with maximal ideal $M$.
The aim of this subsection is to describe a generating set for the kernel of the map $ev_{X=0}^*\colon\overline{\St}(\Phi, A[X], M[X]) \to \overline{\St}(\Phi, A, M)$
induced by the ring homomorphism of evaluation at $0$. We denote this kernel by $K(A[X], M[X])$.

It is obvious that $K(A[X], M[X])$ contains the subgroup $\overline{\St}(\Phi, A[X], XM[X])$.
It turns out that, although $K(A[X], M[X])$ is generally strictly larger than $\overline{\St}(\Phi, A[X], XM[X])$,
 it contains very few extra generators, which all can be explicitly described (see~\cref{Kgen} and the corollary that follows it). 

It follows from~\cref{Kdecomp1} below that $K(A[X], M[X])$ coincides with the commutator subgroup $[\overline{\St}(\Phi, A[X], M[X]), \overline{\St}(\Phi, A[X], XA[X])].$ Thus, if we replace relative Steinberg groups in the statement of~\cref{Kgen} with relative elementary groups, the resulting assertion turns into a special case of a much more general recent result of N.~Vavilov and Z.~Zhang (cf.~\cite[Theorem~1]{VZ18}).

Since $ev_{X=0}^*$ admits a section, we can consider $\overline{\St}(\Phi, A, M)$ and $\St(\Phi, M)$ as subgroups of $\overline{\St}(\Phi, A[X], M[X])$,
 moreover, one has 
\begin{equation} \label{eq:sd-decomp} \overline{\St}(\Phi, A[X], M[X]) = \overline{\St}(\Phi, A, M) \ltimes K(A[X], M[X]).\end{equation}
\begin{lemma} \label{Kdecomp1} The following decomposition holds:
 \[ K(A[X], M[X]) = \overline{\St}(\Phi, A[X], XM[X]) \cdot \left[\St(\Phi, XA[X]),\ \overline{\St}(\Phi, A, M)\right].\] \end{lemma}
\begin{proof} Let us fix $g(X) \in K(A[X], M[X])$. By~\cref{thm:Tits} we can write it as $\prod_i z_{\alpha_i}(f_i(X), \xi_i(X))$ for some $f_i(X) = f_i(0) + Xf_i'(X) \in M[X]$, $\xi_i(X) = \xi_i(0) + X\xi_i'(X) \in A[X]$.
 It is clear that modulo $\overline{\St}(\Phi, A[X], XM[X])$ the element $g(X)$ is congruent to $g_1(X) = \prod_i z_{\alpha_i}(f_i(0), \xi_i(X)).$ 
 
 Now each factor $z_{\alpha_i}(f_i(0), \xi_i(X))$ can be written as follows: 

 \[z_{\alpha_i}(f_i(0), \xi_i(0))^{x_{-\alpha_i}(X\xi'_i(X))} = [x_{-\alpha_i}(-X\xi'_i(X)),\ z_{\alpha_i}(f_i(0), \xi_i(0))] \cdot z_{\alpha_i}(f_i(0), \xi_i(0)).\]
 It follows from the formula $[g,\ h]^{h_1} = [h_1^{-1},\ g][g,\ h_1^{-1}h]$ that the subgroup \[C_0 := \left[\St(\Phi, XA[X]),\ \overline{\St}(\Phi, A, M)\right]\] is normalized by $\overline{\St}(\Phi, A, M)$. Thus, we conclude that $g_1(X)$ is congruent to $\prod_i z_{\alpha_i}(f_i(0), \xi_i(0)) = g(0) = 1$ modulo $C_0$,
 which implies the assertion. \qedhere \end{proof}

Let $S \subseteq \Phi$ be a special root subset (see~\cref{df:root-subsets}). We denote by $U(S, M)$ the subgroup of $\St(\Phi, A)$ generated by root subgroups $X_\alpha(M)$ corresponding to all $\alpha \in S$. We denote by $\Phi^+$ (resp. $\Phi^-$) the subsets of positive (resp. negative) roots of $\Phi$ with respect to some chosen order on $\Phi$.
 
\begin{externaltheorem}[Stein] \label{thm:Stein} One has \[\overline{\St}(\Phi, A, M) = U(\Phi^+, M) \cdot \myol{H}(\Phi, A, M) \cdot U(\Phi^-, M).\] \end{externaltheorem} \begin{proof} See~\cite[Theorem~2.4]{Ste73}. \end{proof}

\begin{prop} \label{Kgen} The subgroup $K(A[X], M[X])$ is generated as an abstract group by the subgroup $\overline{\St}(\Phi, A[X], XM[X])$ and
 the elements $[x_\alpha(m), x_{-\alpha}(X\xi)]$, $m \in M$, $\xi \in A[X]$, $\alpha \in \Phi$. \end{prop}
\begin{proof} From~\eqref{eq:conj-h-x} we obtain that $\myol{H}(\Phi, A, M)$ normalizes both $\St(\Phi, XA[X])$ and $\St(\Phi, M)$ and, moreover, that \[[\myol{H}(\Phi, A, M),\ \St(\Phi, XA[X])] \subseteq \overline{\St}(\Phi, A[X], XM[X]).\]

Denote by $C_1$ the commutator subgroup $[\St(\Phi, XA[X]),\ \St(\Phi, M)]$.
It is clear that for $g \in \St(\Phi, XA[X])$, $h \in \myol{H}(\Phi, A, M)$, $u^+ \in U(\Phi^+, M)$, $u^- \in U(\Phi^-, M)$ one has:
\[ [g,\ h u^+ u^-] = [g,\ h] \cdot [{}^{h}\!g,\ {}^{h}\!(u^+u^-)] \in \St(\Phi, A[X], XM[X]) \cdot C_1.\]
Since $C_0$ is generated by the above commutators and $\overline{\St}(\Phi, A[X], XM[X])$ is a normal subgroup of $\St(\Phi, A[X])$
we obtain that $C_0 \subseteq \overline{\St}(\Phi, A[X], XM[X]) \cdot C_1$ and consequently that
$K(A[X], M[X]) = \overline{\St}(\Phi, A[X], XM[X]) \cdot C_1.$
 
It is clear that modulo $\overline{\St}(\Phi, A[X], XM[X])$ the commutator subgroup $C_1$ is generated by elements of the form $[x_\alpha(m),\ x_{-\alpha}(X\xi)]^g$, where $m \in M$, $\xi \in A[X]$, $g \in \St(\Phi, A[X])$.
Thus, it remains to show that commutators $[[x_\alpha(m),\ x_{-\alpha}(X\xi)],\ g]$ belong to $\overline{\St}(\Phi, A[X], XM[X])$.
Since the latter subgroup is normal it suffices to prove this inclusion in the special case when $g$ is a member of some generating set for $\St(\Phi, A[X])$.
Clearly, the set consisting of $x_\beta(\xi)$, $\xi \in A[X]$, $\beta \neq \pm \alpha$ is such a generating set, 
 and in this case the required inclusions follow from~\ref{C1}--\ref{C3} of~\cref{Crels}. \end{proof}

\begin{corollary} \label{Kgen-strong} For a local pair $(A, M)$ and arbitrary fixed root $\gamma$ of an irreducible simply-laced root system $\Phi$ the subgroup $K(A[X], M[X])$ is generated as a group by $\overline{\St}(\Phi, A[X], XM[X])$ and the elements $c_{\gamma}(m, X\eta)$, where $m \in M$, $\eta \in A[X]$. \end{corollary}
\begin{proof} Substituting $\xi = 1$, $s = m$, $t = X\eta$ into~\cref{Crels}\ref{C4} we obtain that modulo the subgroup
 $\overline{\St}(\Phi, A[X], XM[X])$ the element $c_{\alpha + \beta}(m, X\eta)$ is equivalent to $c_{\alpha}(-\epsilon m, -\epsilon X \eta)^{-1}$, $\epsilon = N_{\alpha, \beta}$. The assertion of the corollary now easily follows from the irreducibility of $\Phi$. \end{proof}   

\section{Proof of the main result}
The main result of this section is~\cref{thm:P1glueing}, which is a direct generalization of~\cite[Proposition~4.3]{Tu83}. 
The object playing a key role in its proof is a certain action of the group $G = \St(\Phi, A[X\inv] + M[X])$ on a certain set $\overline{V}$, which is defined in~\cref{sec:V-construction}. Later, we will see that $\overline{V}$ is, in fact, a set-theoretic $G$-torsor.
To be able to write an explicit formula for this action we need two major ingredients. The first one is ~\cref{lemma33}, which gives a presentation of $G$ with much fewer generations and relations than in the original presentation~\eqref{rel:add}-\eqref{rel:CCF}. The other ingredients are certain subgroups $P_\alpha(0)$, $P_\alpha(*)$ of $\St(\Phi, A[X, X\inv])$ modeled after the nameless groups from~\cite[Lemma~3.4]{Tu83}. The definition and properties of these groups are given in Sections~\ref{sec:Pa0-basic}--\ref{sec:S-a}.

Throughout this section we use the following notations and conventions:
\begin{itemize}
 \item $A$ denotes an arbitrary commutative ring and $M$ is an ideal of $A$. Starting from subsection~\ref{sec:Pa0-basic} we also assume that $A$ is local and $M$ is the maximal ideal of $A$.
 \item We denote by $R$ the Laurent polynomial ring $A[X, X^{-1}]$. We set $t = X^{-1}$, so that $R = A[X, X^{-1}] = A[t, t^{-1}]$.
 \item $B$ denotes the subring $A[t] + M[t^{-1}] = A[X^{-1}] + M[X]$ of $R$.
 \item $I$ denotes the ideal $M[X, X^{-1}]$ of $R$ (it is clear that $I$ is also an ideal of $B$).
 \end{itemize}
\subsection{Presentation of Steinberg groups by homogeneous generators}
\label{sec:presentation}
We consider $R = A[t, t\inv]$ as a $\mathbb{Z}$-graded ring in which $t$ has degree $1$. This grading induces the grading on the subring $B \subseteq R$. For an integer $d$ we denote by $R_d$ (resp. $B_d$) the degree $d$ part of the ring $R$ (resp. $B$). Obviously, $B_d=M \cdot t^d$ for $d<0$, and $B_d=A \cdot t^d$ for $d\geq0$. With this notation, $B$ decomposes as $\oplus_{d\in\mathbb Z}B_d$ as an $A$-module. 

Whenever the coefficient $\xi$ of a Steinberg generator $g = x_\alpha(\xi)$ of $\St(\Phi, B)$ is a homogeneous element of $B$, i.\,e. $\xi \in B_d$ for some $d \in \mathbb{Z}$,
 we call the corresponding generator $g$ {\it homogeneous of degree $d$}.
We denote by $\St^h(\Phi, B)$ the group given by the set of all homogeneous Steinberg generators modulo the following set of Steinberg relations between them
(below $a, a' \in B_d$, $b\in R_e$ and $d,e \in \mathbb{Z}$): 
\begin{align}
&\,\,x_{\alpha}(a)\cdot x_{\alpha}(a') =  x_{\alpha}(a+a'),                        & \tag{R$1_d$} \\
&\,[x_{\alpha}(a),\,x_{\beta}(b)]= x_{\alpha+\beta}(N_{\alpha, \beta} \cdot ab),   & \alpha + \beta \in \Phi,\ \tag{R$2_{d,e}$} \\
&\,[x_{\alpha}(a),\,x_{\beta}(b)]= 1,                                              & \alpha - \beta \in \Phi,\ \tag{R$3^\angle_{d,e}$} \\
&\,[x_{\alpha}(a),\,x_{\beta}(b)]= 1,                                              & \alpha \perp \beta.\ \tag{R$3^\bot_{d,e}$}
\end{align}
\begin{rem}
Notice that in the above presentation we omitted the relations $[x_\alpha(a), x_\alpha(b)] = 1$ but it easy to see that they follow from~$\text{R2}_{e,0}$, \eqref{eq:H1iii} and $\text{R3}_{d,e}^\angle$. Indeed, after choosing some root $\beta\in \Phi$ such that $\alpha+\beta\in\Phi$ and setting $\epsilon = N_{\alpha+\beta,-\beta}$ we obtain that 
\[ [x_\alpha(a),\ x_\alpha(b)] = [x_\alpha(a),\ [x_{\alpha+\beta}(b),\ x_{-\beta}(\epsilon)]] = [[x_\alpha(a),\ x_{\alpha+\beta}(b)],\ {}^{x_{\alpha+\beta}(b)}\!x_{-\beta}(\epsilon)] = 1. \] \end{rem}

It is not hard to show that the map $\St^h(\Phi, B) \to \St(\Phi, B)$ induced by the obvious embedding of generators is an isomorphism. Thus, $\St^h(\Phi, B)$ can be considered as an alternative presentation of $\St(\Phi, B)$ by homogeneous generators.

By the degree of a Steinberg relation we mean the maximum of degrees of generators that appear in the relation.
For example, the degree of every relation of type $\text{R2}_{d,e}$ is $\max(d,e,d+e)$, while the degree of a relation of type $\text{R3}^\bot_{d,e}$ or $\text{R3}^\angle_{d,e}$ is $\max(d,e)$.

For $n\geq 1$ we define the ``truncated'' Steinberg group $\St^{\leq n}(\Phi, B)$ by means of the set $\mathcal{X}_{\leq n}^\Phi$ of homogeneous Steinberg generators of degree $\leq n$ and the subset $\mathcal{R}_{\leq n}^\Phi$ of the above set of Steinberg relations consisting of all relations of degree $\leq n$.  We denote by $F(\mathcal{X}^\Phi_{\leq n})$ the free group on $\mathcal{X}^\Phi_{\leq n}$. Notice that $\varinjlim\limits \St^{\leq n}(\Phi, B) = \St^h(\Phi, B) \cong \St(\Phi, B)$.

The following lemma asserts that most of the relations of type $\text{R3}^\bot_{d,e}$ of positive degree in this presentation of $\St^{\leq n}(\Phi, B)$ 
 are superfluous and can be omitted.
\begin{lemma}\label{superfluous-relations}
 For every simply-laced root system $\Phi$ of rank $\geq 3$ and every $n \geq 1$ one can exclude from
 the presentation of $\St^{\leq n}(\Phi, B)$ all relations of type $\text{R3}_{d,e}^\bot$ whenever $\max(0,d) + \max(0,e) > 1$.
\end{lemma}
\begin{proof}
The proof is based on the following observation: every relation $R$ of type $\text{R3}^\bot_{d,e}$ of degree $\geq 2$ is a consequence of some relation of type $\text{R3}^\bot$ of strictly smaller degree modulo the remaining relations of $\St^{\leq n}(\Phi, B)$, i.\,e. relations of type R2 and $\text{R3}^{\angle}$ (we use the commutator formulae $\text{R2}$ to reduce the degree of the monomials appearing in $R$).
 
Let us fix some relation $R = [x_\alpha(at^d),\ x_\gamma(bt^e)] = 1$ of type $\text{R3}^\bot_{d,e}$ for some $\alpha\perp\gamma$. We can find $\beta \in \Phi$ forming an obtuse angle with both $\alpha$ and $\gamma$ (see e.\,g.~\cite[Lemma~3.1.2]{RS76}).
Without loss of generality we may assume $e \geq d$ and $e > 0$. We need to consider two cases.
\begin{enumerate}
\item In the case $0 < d \leq e \leq n$ the relation $R$ is a consequence of some relation of type $\text{R3}^\bot_{0,e-d}$:
\begin{align} x_{\alpha+\beta+\gamma}(-\epsilon_1 \epsilon_2 \cdot abt^e) = \hspace{4cm} \label{first-computation} \\ =
[x_{\beta+\gamma}(\epsilon_1 \cdot bt^{e-d}), [x_{-\beta}(t^d), x_{\alpha+\beta}(-\delta_1 \cdot a)] ] & \text{ by $\text{R2}_{d,0}$, $\text{R2}_{d, e-d}$} \nonumber \\ 
 = [[x_{\beta+\gamma}(\epsilon_1 \cdot bt^{e-d}), x_{-\beta}(t^d)], {}^{x_{-\beta}(t^d)} x_{\alpha+\beta}(- \delta_1 \cdot a)] & \text{ by~\eqref{eq:H1iii}, $\text{R3}^\bot_{0,e-d}$} \nonumber \\ 
 = {}^{x_{-\beta}(t^d)} [x_{\gamma}(- bt^e), x_{\alpha+\beta}(-\delta_1 \cdot a)] & \text{ by $\text{R2}_{e-d,d}$, $\text{R3}^\angle_{d,e}$} \nonumber \\ 
 = {}^{x_{-\beta}(t^d)} [x_{\beta + \gamma}(\epsilon_1 \cdot bt^{e-d}), x_{\alpha}(at^d)] & \text{ by $\text{R2}_{e,0}$, $\text{R2}_{e-d,d}$} \nonumber \\ 
 = [x_{\gamma}(bt^e) \cdot x_{\beta + \gamma}(\epsilon_1 \cdot bt^{e-d}), x_{\alpha}(at^d)] & \text{ by $\text{R2}_{e-d,d}$, $\text{R3}^\angle_{d,d}$} \nonumber \\ 
 = {}^{x_\gamma(bt^e)}x_{\alpha+\beta+\gamma}(-\epsilon_1\epsilon_2 \cdot abt^e) \cdot [x_\gamma(bt^e), x_\alpha(at^d)] & \text{ by~\eqref{eq:H1ii}, $\text{R2}_{e-d,d}$} \nonumber \\
 = x_{\alpha+\beta+\gamma}(-\epsilon_1\epsilon_2 \cdot abt^e) \cdot [x_\gamma(bt^e), x_\alpha(at^d)] & \text{ by $\text{R3}^\angle_{e,e}$}, \nonumber \end{align}
 where $\epsilon_1 = N_{\beta,\gamma}$, $\epsilon_2 = N_{\alpha,\beta+\gamma}$, $\delta_1 = N_{\alpha,\beta}$ and in the 4th equality we use~\eqref{eq:cocycle}.
  
\item In the case $d \leq 0 \leq e \leq n$ the relation $R$ is a consequence of some relation of type $\text{R3}^\bot_{1,d+e-1}$:
\begin{align*} [x_\alpha(at^d), x_{\gamma}(bt^{e})] = \hspace{4cm} \\ = [x_\alpha(at^d), [x_{\beta+\gamma}(b t^{e-1}), x_{-\beta}(-\epsilon_1 t)]] & \text{ by~$\text{R2}_{e-1,1}$ } \\
 = [[x_\alpha(at^d), x_{\beta+\gamma}(bt^{e-1})], {}^{x_{\beta+\gamma}(bt^{e-1})}\!x_{-\beta}(-\epsilon_1 t)] & \text{ by \eqref{eq:H1iii} and $\text{R3}^\angle_{d,1}$ } \\
 = {}^{x_{\beta+\gamma}(bt^{e-1})}\![x_{\alpha+\beta+\gamma}(\epsilon_2abt^{d+e-1}), x_{-\beta}(-\epsilon_1 t)] & \text{ by $\text{R2}_{d,e-1}$ and $\text{R3}^\angle_{e-1,d+e-1}$  } \\
 = 1 & \text{ by $\text{R3}^\bot_{1,d+e-1}$,} \end{align*}
where $\epsilon_1 = N_{\beta,\gamma}$, $\epsilon_2 = N_{\alpha,\beta+\gamma}$.
\end{enumerate}                                                           

The assertion of the lemma now follows from the above observation by induction on the degree of $R$ and the fact that by~\eqref{first-computation} the relation $\text{R3}^\bot_{1,1}$ is a consequence of $\text{R3}^\bot_{0,0}$.
\end{proof}

The following proposition is the main result of this subsection and also a direct generalization of~\cite[Lemma~3.3]{Tu83}.
\begin{prop} \label{lemma33} For $\Phi=\rA_{\geq 4}, \rD_{\geq 5}, \rE_{6,7,8}$ and $n \geq 1$ the homomorphism $i_n\colon \St^{\leq n}(\Phi, B) \to \St^{\leq n+1}(\Phi, B)$, induced by the natural embedding of generators, is an isomorphism. Consequently, the obvious homomorphism $\St^{\leq 1}(\Phi, B) \to \St(\Phi, B)$ is an isomorphism. \end{prop}
\begin{proof}
 We need to construct a homomorphism $j_n$ which would be the inverse of $i_n$. 
 We start with a homomorphism $\widetilde{j}_n^\Phi \colon F\langle \mathcal{X}^\Phi_{\leq n+1} \rangle \to \St^{\leq n}(\Phi, B)$ defined via
 \[ \widetilde{j}^{\Phi}_n(x_\alpha(at^k)) = \begin{cases} x_\alpha(at^k), & k\leq n; \\
      [x_{\alpha - \beta} (N_{\alpha-\beta, \beta} \cdot at^{k-1}), x_{\beta}(t)], & k = n+1, \end{cases} \]
 where $\beta$ is any root of $\Phi$ forming an acute angle with $\alpha$.
 A standard argument, cf. \cite[Proposition~1.1]{Re75} or~\cite[Proposition~3.2.2]{RS76}, shows that $\widetilde{j}^\Phi_n$ does not depend on the choice of $\beta$.
  
 Set $\mathcal{R}^\Phi_{n+1} = \mathcal{R}^\Phi_{\leq n+1} \setminus \mathcal{R}^\Phi_{\leq n}$. It suffices to verify that the image of every relation $R \in \mathcal{R}^\Phi_{n+1}$ under $\widetilde{j}^\Phi_n$ is a trivial element of $\St^{\leq n}(\Phi, B)$. In the special case $\Phi=\rA_{\geq 4}$ this has already been demonstrated by Tulenbaev in~\cite[Lemma~3.3]{Tu83}, so in this case the proof of the proposition is complete. We will deduce the assertion in the remaining cases $\Phi=\rD_\ell,\rE_\ell$ from the special case $\Phi=\rA_4$ of Tulenbaev's result.
 
 Let $R$ be a relation from $\mathcal{R}^\Phi_{n+1}$. By~\cref{superfluous-relations} we may assume that $R$ is not of type $\text{R3}^\bot$, therefore the roots $\alpha, \beta$ appearing in $R$ are contained in a root subsystem of $\Phi$ of type $\rA_2$. Our assumptions on $\Phi$ guarantee that there exists some root subsystem $\Psi$ of type $\rA_4$ containing $\alpha$ and $\beta$. Consider the following commutative diagram in which the vertical arrows are induced by the embedding $\Psi\subseteq\Phi$. 
 \[\begin{tikzcd} F(\mathcal{X}_{\leq n+1}^\Psi) \arrow{r}{j_n^\Psi} \arrow{d} & \St^{\leq n}(\Psi, B) \arrow{d} \\ F(\mathcal{X}_{\leq n+1}^\Phi) \arrow{r}{j_n^\Phi} & \St^{\leq n}(\Phi, B) \end{tikzcd} \]  
The relation $R$ lies in the image of the left arrow, therefore it comes from some relation $R' \in \mathcal{R}^\Psi_{n+1}$. The image of $R'$ in $\St^{\leq n}(\Psi, B)$ under $j_n^\Psi$ is trivial by Tulenbaev's result. But this implies that the image of $R$ under $\widetilde{j}_n^\Phi$ is also trivial and hence that $\widetilde{j}_n^\Phi$ gives rise to the desired map $j_n$.
\end{proof}

\begin{rem}
 Notice that in the case $\Phi=\rD_\ell$ the pair $\{\alpha_{\ell-1}, \alpha_{\ell}\}$ of orthogonal simple roots cannot be embedded into a root subsystem of type $\rA_4$. This explains why we needed to exclude relations $\text{R3}^\bot$ from the presentation of $\St^{\leq n}(\Phi, B)$ in the proof of the above proposition.
\end{rem}
\begin{rem}
 In the special case $M=A$, $B = A[t, t\inv]$ the assertion of the above proposition also holds in the cases $\Phi=\rA_2, \rA_3, \rD_4$. This is a consequence of the presentation of D.~Allcock applied to the affine untwisted Steinberg group $\St(\Phi, A[t, t\inv]) \cong \St(\widetilde{\Phi}^{(1)}, A)$. Allcock's presentation implies that $\St(\Phi, A[t, t\inv])$ can be presented using only generators and relations of degree $\leq 1$ with respect to both $t$ and $t^{-1}$, see~\cite[Corollary~1.3]{A13}.
 
 In the cases $\Phi = \rA_3, \rD_4$, $M \neq A$ it is still possible to prove the injectivity of $i_n$ starting from $n\geq 2$ using a variation of the argument of Rehmann--Soul{\'e} (cf. the lower bound for $m$ in 3.2.1 of~\cite{RS76}). However, apparently, it is not possible to establish the injectivity of $i_1$ in the specified cases using arguments similar to~\cite{RS76}.
\end{rem}

\subsection{The subgroups \texorpdfstring{$P_\alpha(0)$}{Pa(0)}, \texorpdfstring{$P_\alpha(*)$}{Pa(*)}, \texorpdfstring{$K(\alpha, \beta)$}{K(a,b)} and their properties} \label{sec:Pa0-basic}
Let $\Phi$ be a root system and $\alpha$ be an element of $\Phi$.
Consider the following subsets of $\Phi$:
\begin{align} Z_+(\alpha) & = \{ \beta \in \Phi \mid \langle \alpha, \beta \rangle > 0 \}, \\
   Z_0(\alpha) & = \{ \beta \in \Phi \mid \alpha + \beta \not\in \Phi,\ \langle \alpha, \beta \rangle = 0 \}, \\
   Z(\alpha)   & = Z_0(\alpha) \sqcup Z_+(\alpha). \end{align}
Clearly, $Z_0(\alpha)$ (resp. $Z_+(\alpha)$) is a symmetric (resp. special) subset of $\Phi$ (see~\cref{df:root-subsets} for the terminology).
   
We denote by $Z_\alpha(A, M)$ the subgroup of $\overline{\St}(\Phi, A, M)$ generated by elements
 $x_{\beta}(m),\ \beta \in Z_+(\alpha)$ and $z_{\gamma}(m, \zeta),\ \gamma \in Z_0(\alpha),$ where $m \in M$, $\zeta \in A$ (see the beginning of~\cref{sec:elementary} for the definition of the elements $z_\gamma(m, \zeta)$ and the group $\overline{\St}(\Phi, A, M)$).
It is not hard to see that \[Z_\alpha(A, M) = \Img\left(\overline{\St}(Z_0(\alpha), A, M) \to \overline{\St}(\Phi, A, M)\right) \ltimes U(Z_+(\alpha), M). \]
In the above formula $U(Z_+(\alpha), M)$ denotes the subgroup defined before~\cref{thm:Stein}.
It is clear, that $Z_\alpha(A, M)$ centralizes the root subgroup $X_\alpha(A)$ (cf.~\cite[984]{St71}).

For the rest of this subsection $\Phi$ is a simply-laced root system of rank $\geq 3$.
Let $\alpha$ be a fixed root of $\Phi$. 
Notice that in the simply-laced case the assumption $\alpha+\beta\not\in \Phi$ in the definition of $Z_0(\alpha)$ is superfluous, i.\,e. $Z_0(\alpha) = \{ \beta\in\Phi \mid \alpha \perp \beta \}$
(cf.~\cite[Proposition~5.7]{St71}).

\begin{rem}\label{Z-DS} Notice that our assumptions on the rank of $\Phi$ guarantee that $Z_0(\alpha)$ is nonempty. In particular, if $A$ is a local ring with maximal ideal $M$ the group $Z_\alpha(A, M)$ contains relative Dennis--Stein symbols $\langle a, m \rangle$ for $a\in A$, $m\in M$.
By~\eqref{DS-S-relationship} relative Steinberg symbols $\{a, 1+m\}$ are also contained in $Z_\alpha(A, M)$ for all $a\in A^\times$, $m\in M$. \end{rem}

\begin{df}\label{defP0}
Let $M$ be an ideal of a local ring $A$. Denote by $\widetilde{P}_{\alpha, M}(0)$ the subgroup of $\overline{\St}(\Phi, A[X], M[X])$ generated by the following 5 families of elements parameterized by $f \in M[X]$, $\xi \in A[X]$:
\begin{enumerate}[label=(P\arabic*)]
 \item $z_{\beta}(Xf, \xi),\ \beta \in \Phi \text{ such that }\alpha + \beta \in \Phi\text{ or } \alpha - \beta \in \Phi;$
 \item $z_{\beta}(f, X\xi),\ \alpha - \beta \in \Phi;$
 \item $z_{\beta}(f, \xi),\ \beta \perp \alpha;$
 \item $x_{-\alpha}(X^2f)$;
 \item $x_{\alpha}(f)$. \end{enumerate}

We denote by $\widetilde{P}_{\alpha, M}(*)$ the subgroup of $\overline{\St}(\Phi, A[X], M[X])$ generated by $\widetilde{P}_{\alpha, M}(0)$ and the elements $x_{-\alpha}(mX)$, $m \in M$.
\end{df}
Almost always we will be using the above definition in the situation when $M$ is precisely the maximal ideal of $A$. The only exception to this is~\cref{P0-conj-invariant} where the above subgroups are also used for $M=A$.

In the sequel we continue using letters $f$ and $\xi$ to denote elements of $M[X]$ and $A[X]$, respectively.

\begin{rem} \label{rem:recognition} From the above definition it follows that an element $z_\gamma(Xf, \xi)$ is a generator of type P1 provided $\gamma\neq\pm\alpha$ and $\gamma \not \perp \alpha$. Since $x_\gamma(Xf) = z_\gamma(Xf, 0)$, it follows that $x_{\gamma}(Xf)$ is also a generator of type P1. In particular, if $\beta$ is such that $\alpha + \beta \in \Phi$ then $x_{\pm(\alpha+\beta)}(Xf)$ is a generator of type P1 and $x_{\alpha+\beta}(f)$ is a generator of type P2. \end{rem}

\begin{lemma}\label{P0_normal} The subgroup $\widetilde{P}_{\alpha, M}(0)$ is normal in $\widetilde{P}_{\alpha, M}(*)$. Moreover, there is a short exact sequence of groups, which is split by the map $m \mapsto x_{-\alpha}(mX)$ (we denote by $(M, +)$ the additive group of the ideal $M$):
\[\begin{tikzcd} 1 \arrow{r} & \widetilde{P}_{\alpha,M}(0) \arrow{r} & \widetilde{P}_{\alpha,M}(*) \ar[r, "p_\alpha", twoheadrightarrow] & (M, +) \arrow[l, dashrightarrow, shift left=2] \arrow{r} & 1. \end{tikzcd} \] \end{lemma}
\begin{proof} We need to verify that the conjugate by $x_{-\alpha}(mX)$ of every generator $g$ of $\widetilde{P}_\alpha(0)$ listed in~\cref{defP0} belongs to $\widetilde{P}_\alpha(0)$.
The assertion is obvious for the generators of type P3 and P4.

Let $\beta\in\Phi$ be such that $\alpha - \beta \in \Phi$ and $g = z_\beta(Xf, \xi)$ be a generator of type P1. By~\cref{Zrels}\ref{Z2} we get that \begin{equation} z_{\beta}(Xf, \xi) ^ {x_{-\alpha}(mX)} =  x_{-\alpha} (- mX^2f\xi) \cdot x_{\beta-\alpha} (N_{\beta, -\alpha}\cdot mX^2f) \cdot z_{\beta}(Xf, \xi). \label{eq3-1} \end{equation}
The expression in the right-hand side is a product of generators of type P4, P1, P1 (see~\cref{rem:recognition}).
Now consider the case when $g = z_\beta(f, X\xi)$ is a generator of type P2 (so that $\alpha-\beta\in\Phi$). Again by~\cref{Zrels}\ref{Z2}
\begin{equation}
z_{\beta}(f, X\xi) ^ {x_{-\alpha}(mX)} =  x_{-\alpha} (- mX^2f\xi ) \cdot x_{\beta-\alpha} (N_{\beta, -\alpha}\cdot mXf) \cdot z_{\beta}(f, X\xi). \label{eq3-2} \end{equation}
It is easy to see that the expression in the right-hand side of~\eqref{eq3-2} is a product of generators of type P4, P1, P2.

Now let $g = z_{\beta}(Xf, \xi)$ be a generator of type P1 in the case $\alpha + \beta \in \Phi$. By~\cref{Zrels}\ref{Z3}
\begin{equation} \label{eq3-3} z_{\beta}(Xf, \xi) ^ {x_{-\alpha}(mX)} = x_{-\alpha} (mX^2f\xi ) \cdot x_{-\alpha-\beta} (N_{\beta,\alpha}\cdot mX^2f\xi^2) \cdot z_{\beta}(Xf, \xi), \end{equation}
and the latter expression is a product of generators of type P4, P1, P1.

Finally, let $g = x_\alpha(f)$ be a generator of type P5. Substituting $s = -\epsilon f$, $\xi = -\epsilon m$, $\eta=X$ into the identity of~\cref{Zrels}\ref{Z5} and expressing the fifth factor in the right-hand side of it through other terms we obtain that
\begin{multline} \label{eq:zalpha} g^{x_{-\alpha}(mX)} = z_\alpha(f, mX) = x_{\alpha+\beta}(\epsilon Xf) \cdot x_{\beta}(-mX^2 f) \cdot x_{-\beta}(mf) \cdot x_\alpha(f) \cdot \\ 
 \cdot z_{\alpha+\beta}(-\epsilon X f, -\epsilon m) \cdot z_{-\beta}(-mf, -X) \cdot x_{-\alpha-\beta}(-\epsilon m^2X f) \cdot x_{-\alpha}(-m^2X^2 f), \end{multline}
where $\epsilon = N_{\alpha, \beta}$. It is clear that the latter expression is a product of generators of type P1, P1, P2, P5, P1, P2, P1, P4. \end{proof}

\begin{rem}\label{rem:c} By~\eqref{Cdef} one has $c_\beta(f, X\xi) = x_{\beta} (f) \cdot z_{\beta}(-f, -X\xi)$. This implies that for $\beta \in Z_+(\alpha)$ (resp. $\beta \in Z_0(\alpha)$) the element $c_\beta(f, X\xi)$ is a product of two generators of type P2 (resp. P3).
Thus, the elements $c_{\beta}(f, X\xi)$ lie in $\widetilde{P}_{\alpha, M}(0)$ for all $\beta \in Z(\alpha)$. \end{rem}
\begin{rem}\label{rem:DS}
It is easy to check that $\widetilde{P}_{\alpha, M}(0)$ contains the image of $Z_\alpha(A, M)$ under the natural embedding of $\overline{\St}(\Phi, A, M) \hookrightarrow \overline{\St}(\Phi, A[X], M[X])$.
In particular, if $M$ is the maximal ideal of $A$, the subgroup $\widetilde{P}_{\alpha, M}(0)$ contains relative Dennis--Stein and Steinberg symbols. It is also easy to see that the image of $\widetilde{P}_{\alpha, M}(0)$ under the homomorphism $ev_{X=0}^*$ coincides with $Z_\alpha(A, M)$. \end{rem}

The following lemma shows that $\widetilde{P}_{\alpha, M}(*)$ is sufficiently large.
\begin{lemma} \label{Pstar-large} Suppose that $(A, M)$ is a local pair. 
Then the subgroup $\widetilde{P}_{\alpha, M}(*)$ contains the subgroup $K(A[X], M[X]) \leq \overline{\St}(\Phi, A[X], M[X])$ defined at the beginning of~\cref{sec:computationOfK}. \end{lemma}
\begin{proof} Clearly, $\widetilde{P}_{\alpha, M}(*)$ contains the elements $x_\beta(Xf)$ for all $\beta \in \Phi$ and $z_\beta(Xf, \xi)$ for $\beta \in \Phi \setminus \{\pm \alpha\}$, $f \in M[X]$, $\xi\in A[X]$, therefore by~\cref{thm:Stepanov} $\widetilde{P}_{\alpha, M}(*)$ contains all of $\overline{\St}(\Phi, A[X], XM[X])$. The required assertion now follows from~\cref{Kgen-strong} and~\cref{rem:c}. \end{proof}

\begin{rem} \label{Pstar-char} From the above lemma,~\eqref{eq:sd-decomp} and~\cref{rem:DS} it follows that the subgroup $\widetilde{P}_{\alpha, M}(*)$ admits decomposition $\widetilde{P}_{\alpha, M}(*) = Z_\alpha(A, M) \ltimes K(A[X], M[X]).$ \end{rem} 

\begin{rem} \label{rem:palpha} It follows from~\cref{P0_normal} and~\cref{Pstar-large} that the value of the function $p_\alpha$ from the statement of~\cref{P0_normal} on an element $g\in K(A[X], M[X])$ can be computed via the following procedure. 
Start with any presentation of~$g$ as a product of elements $z_{\beta}(Xf, \xi)$ for $\beta \in \Phi$ and $c_\delta(f, X\xi)$ for some fixed $\delta\in Z_0(\alpha)$. 
Now pick among these factors those that correspond to the root $\beta = -\alpha$ (i.\,e. pick all factors $z_{-\alpha}(Xf_i, \xi_i)$). 
Now $p_\alpha(g)$ is precisely the sum of constant terms of the polynomials $f_i$. \end{rem}

\begin{lemma} \label{P0-conj-invariant} Suppose $(A, M)$ is a local pair. Then for any $\beta \in Z(\alpha)$ and $b \in A$ the subgroups $\widetilde{P}_{\alpha, M}(0)$ and $\widetilde{P}_{\alpha, M}(*)$ are stable under conjugation by $x_\beta(b)$. \end{lemma}
\begin{proof}
Notice that both $Z_\alpha(A, M)$ and $K(A[X], M[X])$ are stable under the specified conjugation, which implies the assertion for $\widetilde{P}_{\alpha, M}(*)$. 
To obtain the assertion for $\widetilde{P}_{\alpha, M}(0)$ consider the following commutative diagram:
\[\begin{tikzcd} \widetilde{P}_{\alpha, M}(0) \arrow[r, hookrightarrow] \arrow[d, hookrightarrow] & \widetilde{P}_{\alpha, M}(*) \arrow[r, "p_{\alpha, M}", twoheadrightarrow] \ar[d, hookrightarrow] & (M, +) \ar[d, hookrightarrow] \\ \widetilde{P}_{\alpha, A}(0) \ar[r, hookrightarrow] & \widetilde{P}_{\alpha, A}(*) \ar[r, "p_{\alpha, A}", twoheadrightarrow] & (A, +). \end{tikzcd} \]
Notice that $x_\beta(b) \in \widetilde{P}_{\alpha, A}(0)$, therefore for $g \in \widetilde{P}_{\alpha, M}(0)$ one has \[p_{\alpha, M}(x_\beta(b) \cdot g \cdot x_\beta(-b)) = p_{\alpha, A}(x_\beta(b) \cdot g \cdot x_\beta(-b)) = p_{\alpha, A}(g) = 0,\] which implies the assertion.\end{proof}

For the rest of this subsection $M$ is the maximal ideal $A$.
We denote by $\beta$ another fixed root of $\Phi$ forming an acute angle with $\alpha$ (i.\,e. $\langle \alpha, \beta \rangle = 1$). We denote by $\Psi$ the subsystem of type $\rA_2$ generated by $\alpha$ and $\beta$.
\begin{df}\label{def:Kab} 
Fix a root $\delta \in \Phi \setminus \Psi$. Define $\widetilde{K}(\alpha, \beta)$ as the subgroup of $\overline{\St}(\Phi, A[X], M[X])$ generated by the following 5 families of elements parameterized by $\xi \in A[X]$, $f\in M[X]$:
 \begin{enumerate}[label=(K\arabic*)]
  \item $z_\gamma(Xf, \xi)$, for all $\gamma \in \Phi \setminus \Psi$;
  \item $x_{-\alpha}(X^2f)$, $x_{-\beta}(X^2f)$;
  \item $x_{\alpha}(Xf)$, $x_\beta(Xf)$;
  \item $x_{\alpha-\beta}(Xf)$, $x_{\beta-\alpha}(Xf)$;
  \item $c_{\delta}(f, X\xi)$. \end{enumerate} \end{df}
\begin{prop} \label{K-a-b} One has $K(A[X], M[X]) = \widetilde{K}(\alpha, \beta) \cdot X_{-\alpha}(M \cdot X) \cdot X_{-\beta}(M \cdot X).$
\end{prop} 
\begin{proof}
It is clear that $K(A[X], M[X]) \supseteq \widetilde{K}(\alpha, \beta) \cdot X_{-\alpha}(M \cdot X) \cdot X_{-\beta}(M \cdot X)$. Let us prove the reverse inclusion. Fix an element $g\in K(A[X], M[X])$. We need to show that $g$ lies in the specified product of subgroups.

For $i=1,\ldots,5$ denote by $\mathcal{G}_i$ the set of all generators of type K$i$ from~\cref{def:Kab}.
Notice that $\Phi \setminus \Psi$ contains the parabolic set of roots $(\Phi \setminus \Psi)\cap \Phi^+$.
Thus, by~\cref{thm:Stepanov} every element of $\overline{\St}(\Phi, A[X], XM[X])$ can be presented as a product of generators $\mathcal{G}_1$ and generators $\mathcal{G}' = \{ x_{\gamma}(Xf) \mid \gamma \in \Psi\}$. In turn, by~\cref{Kgen-strong} one can express $g$ as a product of these generators and generators $\mathcal{G}_5$. 

Set $G_0 = \overline{\St}(\Phi, A[X], X^2M[X])$. By~\cref{thm:Stepanov} $\widetilde{K}(\alpha, \beta)$ contains the normal subgroup $G_0$ therefore it suffices to obtain the required presentation of $g$ modulo $G_0$. By~\cref{Zrels} and~\cref{Crels} the elements of $\mathcal{G}'$ commute with the elements of $\mathcal{G}_1$ and $\mathcal{G}_5$ modulo $G_0$, therefore we can rewrite $g$ as $g_1 \cdot g_2$, for some $g_1 \in \langle \mathcal{G}_1, \mathcal{G}_5, G_0 \rangle \subseteq \widetilde{K}(\alpha, \beta)$ and $g_2 \in \langle \mathcal{G}' \rangle$.

It is clear that  $x_{\gamma_1}(mX)$ and $x_{\gamma_2}(m'X)$ commute modulo $G_0$ whenever $\gamma_1 \neq - \gamma_2$ for $\gamma_1, \gamma_2 \in \Psi$. On the other hand, for $\gamma \in \Psi$ any commutator $c_{\gamma}(mX, m'X)$ is congruent to some generator of $\mathcal{G}_5$  modulo $G_0$ (cf. the proof of~\cref{Kgen-strong}) and hence commutes with the elements of $\mathcal{G}'$ modulo $G_0$. 

Now pick each factor $x_{\gamma}(Xf)$, $\gamma \in \{ -\alpha, -\beta \}$ appearing in the presentation for~$g_2$, decompose it as $x_{\gamma}(X^2f') \cdot x_{\gamma}(m'X)$ for some $f' \in M[X]$, $m' \in M$ and then move the second factor to the rightmost position within $g_2$ conjugating all factors along the way. By the previous paragraph only elements of  $\mathcal{G}_5$ may appear modulo $G_0$ after simplification of these conjugates.
Thus, we can rewrite $g_2$ as $g_{21} \cdot g_{22}$, where $g_{21} \in \langle \mathcal{G}_2, \mathcal{G}_3, \mathcal{G}_4, \mathcal{G}_5, G_0 \rangle \subseteq \widetilde{K}(\alpha, \beta)$ and $g_{22} \in \langle X_{-\alpha}(MX), X_{-\beta}(MX) \rangle = X_{-\alpha}(MX) \cdot X_{-\beta}(MX)$.
Thus, we have obtained the desired decomposition for $g$. \end{proof}

\begin{corollary} \label{K-a-b-cor}
One has $\widetilde{K}(\alpha, \beta) = K(A[X], M[X]) \cap \widetilde{P}_{\alpha, M}(0) \cap \widetilde{P}_{\beta, M}(0)$. In particular, $\widetilde{K}(\alpha, \beta)$ does not depend on the choice of the root $\delta$ in~\cref{def:Kab}.
\end{corollary}
\begin{proof} Denote $\widetilde{P}_{\alpha, M}(0) \cap \widetilde{P}_{\beta, M}(0) \cap K(A[X], M[X])$ by $K$. It is clear that $\widetilde{K}(\alpha, \beta) \subseteq K$. 
Let us prove the reverse inclusion. Let $g$ be an element of $K$.
By~\cref{K-a-b} $g$ can be presented as $g_0 \cdot x_{-\alpha}(mX) \cdot x_{-\beta}(m'X)$ for some $g_0 \in \widetilde{K}(\alpha, \beta)$ and $m, m' \in M$. From~\cref{rem:palpha} we obtain that $0 = p_\alpha(g) = m$, $0 = p_\beta(g) = m'$, therefore $g = g_0$ lies in $\widetilde{K}(\alpha, \beta)$, as required.\end{proof}

\begin{corollary} \label{K-a-b-cor2}
 One has $\widetilde{P}_{\alpha, M}(0) \cap K(A[X], M[X]) = \widetilde{K}(\alpha, \beta) \cdot X_{-\beta}(M \cdot X).$ 
\end{corollary}
\begin{proof} Set $L = \widetilde{P}_{\alpha, M}(0) \cap K(A[X], M[X])$.
The elements of $X_{-\beta}(M\cdot X)$ are generators of type P1 for $\widetilde{P}_{\alpha, M}(0)$ (cf.~\cref{rem:recognition}),
 therefore $X_{-\beta}(M \cdot X) \subseteq L$. On the other hand, by~\cref{K-a-b-cor} $\widetilde{K}(\alpha, \beta) \subseteq L$. Thus, we have shown that $\widetilde{K}(\alpha, \beta) \cdot X_{-\beta}(M \cdot X) \subseteq L$. The reverse inclusion immediately follows from~\cref{rem:palpha} and~\cref{K-a-b}.
 \end{proof}
 
\begin{lemma} \label{conj-K-a-b} For $b \in A$ one has $\widetilde{K}(\alpha, \beta)^{x_{\beta - \alpha}(b)} \subseteq \widetilde{K}(\alpha, \beta)$. \end{lemma}
\begin{proof} By~\cref{Zrels} the conjugate of a generator $g = z_\gamma(Xf, \xi)$ of type K1 by $x_{\beta-\alpha}(b)$ is either $g$ itself or a product of generators of type K4, K1, K1. By~\cref{Crels} the conjugate of a generator $g = c_\delta(f, X\xi)$ of type K5 is either $g$ or a product of generators of type K5, K4, K1.

Now if $g = x_\gamma(Xf)$, $\gamma \in \{\pm \alpha, \pm \beta\}$ is a generator of type K2 or K3, the conjugation by $x_{\beta-\alpha}(a)$ either fixes $g$ or transforms it into a product of two generators of type K2 or K3.  Finally, from~\cref{rem:palpha} we obtain that $z_{\alpha-\beta}(Xf, b) \in K(A[X], M[X]) \cap \widetilde{P}_\alpha(0) \cap \widetilde{P}_{\beta}(0) = \widetilde{K}(\alpha, \beta)$. \end{proof}  

\begin{df} Define the subgroups $P_\alpha(0)$, $P_{\alpha}(*)$, $K(\alpha, \beta) \leq \overline{\St}(\Phi, R, I)$ as the images of the subgroups $\widetilde{P}_{\alpha, M}(0)$, $\widetilde{P}_{\alpha, M}(*)$, $\widetilde{K}(\alpha, \beta)$ under the natural homomorphism $j_+ \colon \St(\Phi, A[X]) \to \St(\Phi, R)$.\end{df}

We need one more technical definition. Denote by $Z_{\alpha, \beta}$ the subgroup $U(Z_+(\alpha) \setminus \{\alpha - \beta \}, M[X])$ of $\St(\Phi, A[X, X^{-1}])$.  Recall that this means that $Z_{\alpha, \beta}$ is the group generated by all root subgroups $X_\gamma(M[X])$, where $\gamma \in Z_+(\alpha) \setminus \{ \alpha - \beta \}$. 
  
\begin{lemma} \label{image-K-a-b} The image of $K(\alpha, \beta)$ under the automorphism of conjugation by $x_\alpha(aX^{-1})$ is contained in the subgroup of $\St(\Phi, A[X, X^{-1}])$ generated by $K(\alpha, \beta)$ and $Z_{\alpha, \beta}$. \end{lemma}
\begin{proof} We need to verify that the conjugate by $x_\alpha(aX^{-1})$ of every generator from~\cref{def:Kab} lies in the specified subgroup. To simplify notation we call the generators of  $Z_{\alpha,\beta}$ ``generators of type Z``. 

Let $z_\gamma(Xf, \xi)$ be a generator of type K1 for some $\gamma \in \Phi \setminus \Psi$.
In the case $\gamma \perp \alpha$ this generator commutes with $x_\alpha(aX^{-1})$ by~\cref{Zrels}\ref{Z4}.
If $\gamma$ is such that $\alpha + \gamma \in \Phi$ we obtain from~\cref{Zrels}\ref{Z2} that
\begin{equation} z_{\gamma}(Xf, \xi) ^ {x_{\alpha}(aX^{-1})} = x_{\alpha} (- af\xi) \cdot x_{\alpha+\gamma} (N_{\gamma, \alpha}\cdot af) \cdot z_{\gamma}(Xf, \xi). \end{equation}
On the other hand, if $\gamma$ is such that $\alpha - \gamma \in \Phi$, we obtain from~\cref{Zrels}\ref{Z3} that
\begin{equation} z_{\gamma}(Xf, \xi) ^ {x_{\alpha}(aX^{-1})} = x_{\alpha} (af\xi) \cdot x_{\alpha-\gamma} (N_{\gamma,-\alpha}\cdot af\xi^2) \cdot z_{\gamma}(Xf, \xi). \end{equation}
By the choice of $\gamma$ both $\alpha + \gamma$ and $\alpha - \gamma$ lie in $Z_+(\alpha) \setminus \{\alpha - \beta\}$, therefore in both cases the expressions in the right-hand side are products of generators of type Z, Z, K1.

A similar computation using \ref{C1}--\ref{C3} of \cref{Crels} shows that the conjugate of a generator of type K5 is either the generator itself or a product of generators of type Z, Z, K5.

Let us verify the assertion for the generators of types K2, K3, K4.
By~\eqref{Chevalley-CCF2} the conjugation by $x_\alpha(aX^{-1})$ fixes root subgroups $X_{\alpha}(X\cdot M[X])$, $X_{\beta}(X\cdot M[X])$, $X_{\alpha-\beta}(X\cdot M[X])$. 
Further, from~\eqref{Chevalley-CCF1} we obtain that
\begin{align*}
 x_{-\beta}(X^2f)^{x_\alpha(aX^{-1})}     &= x_{-\beta}(X^2f) \cdot x_{\alpha-\beta}(N_{-\beta, \alpha} \cdot aXf);\\
 x_{\beta-\alpha}(Xf)^{x_\alpha(aX^{-1})} &= x_{\beta-\alpha}(Xf) \cdot x_{\beta}(N_{\alpha,-\beta}\cdot af),
\end{align*}
and the expressions in the right-hand side are products of generators of type K2, K4 and K4, Z, respectively.

It remains to verify the assertion for the generators $x_{-\alpha}(X^2f) $ of type K2.
By~\cref{K-a-b-cor} we may assume that the root $\delta$ in~\cref{def:Kab} forms an acute angle with $\alpha$, i.\,e. $\alpha-\delta \in \Phi$. Substituting $s = Xf$, $\eta = X$, $\xi = aX^{-1}$ and $\alpha = \delta-\alpha$, $\beta = -\delta$ into~\cref{Zrels}\ref{Z5} we obtain that
\begin{multline*} z_{-\alpha}(X^2f, aX^{-1}) = x_{\delta-\alpha}(\epsilon Xf) \cdot x_{\delta}(-af) \cdot x_{-\delta}(aX^2 f) \cdot x_{-\alpha}(X^2f) \cdot \\
 \cdot z_{\delta-\alpha}(-\epsilon Xf, -\epsilon a) \cdot x_{\alpha-\delta}(-\epsilon a^2 Xf) \cdot x_{\alpha}(- a^2 f) \cdot z_{\delta}(a f, -X), \end{multline*}
where $\epsilon = N_{-\delta,\alpha}$. 
The first 7 factors in the right-hand side are generators of type K1, Z, K1, K2, K1, K1, Z.
The remaining last factor $z_\delta(af, -X)$ can be rewritten as $x_{\delta}(af) \cdot c_{\delta}(-af, X)$ (cf.~\cref{rem:c}) and hence is a product of generators of type Z and K5. \end{proof}

\begin{corollary} \label{K-a-b-P-b} The image of $K(\alpha, \beta)$ under the automorphism of conjugation by $x_\alpha(aX^{-1})$ is contained in $P_\beta(0)$. \end{corollary}
\begin{proof} 
By~\cref{K-a-b-cor} and~\cref{image-K-a-b} it suffices to show that $Z_{\alpha, \beta} \subseteq P_\beta(0)$. Let $g = x_{\gamma}(f)$ be a generator of $Z_{\alpha, \beta}$ for some
$\gamma \in Z_+(\alpha) \setminus \{\alpha - \beta \}$. Notice that $\gamma$ cannot form an obtuse angle with $\beta$, otherwise from $\langle \beta, \gamma \rangle = -1$ it follows that $ \langle \alpha - \beta, \gamma \rangle \geq 2$ and hence that $\gamma = \alpha - \beta$, a contradiction.
Now if $\langle \beta, \gamma \rangle$ is $0$, $1$ or $2$, then $g$ is a generator of type P3, P2 or P5 for $P_\beta(0)$, respectively.\end{proof}

\begin{lemma} \label{P0_conj} The subgroup $P_\alpha(0)$ is stable under conjugation by $x_\alpha(aX^{-1})$ for arbitrary $a \in A$. \end{lemma}
\begin{proof} 
Intersecting the factors of~\eqref{eq:sd-decomp} with $\widetilde{P}_{\alpha, M}(0)$ and invoking~\cref{rem:DS} and~\cref{K-a-b-cor2} we obtain that $P_{\alpha}(0) =  j_+(Z_\alpha(A, M)) \ltimes \left(K(\alpha, \beta) \cdot X_{-\beta}(M \cdot X)\right).$
It suffices to show that the image of each of these three subgroups under the automorphism of conjugation by $x_\alpha(aX^{-1})$ is contained in $P_\alpha(0)$.

It is clear that the specified conjugation fixes $j_+(Z_\alpha(A, M))$. By~\eqref{Chevalley-CCF1} and~\cref{rem:recognition} 
$X_{-\beta}(M\cdot X)^{x_\alpha(aX^{-1})} \subseteq X_{-\beta}(M\cdot X) \cdot X_{\alpha-\beta}(M) \subseteq P_\alpha(0).$ 
Notice that the generators of $Z_{\alpha,\beta}$ are generators of type P2 for $P_\alpha(0)$.
Thus, from~\cref{image-K-a-b} we get that
$K(\alpha, \beta)^{x_\alpha(aX^{-1})} \subseteq \langle K(\alpha, \beta), Z_{\alpha, \beta} \rangle \subseteq P_\alpha(0)$.
\end{proof}

\begin{lemma} \label{conj-S-a-x-bma} For $a, b\in A$ the image of $K(\alpha,\beta)^{x_\alpha(aX^{-1})}$ under the automorphism of conjugation by $x_{\beta-\alpha}(b)$ is contained in $P_\alpha(0)$. \end{lemma}
\begin{proof} By Lemmas~\ref{conj-K-a-b}, \ref{image-K-a-b} it suffices to show that the conjugate $g_1$ of each generator $g = x_\gamma(f)$, $\gamma \in Z_+(\alpha) \setminus \{ \alpha - \beta \}$ of $Z_{\alpha, \beta}$ by $x_{\beta-\alpha}(b)$ belongs to $P_\alpha(0)$. It is clear that $g$ itself is a generator of type P2 for $P_\alpha(0)$, therefore it remains to consider the case $\gamma\not \perp \beta - \alpha$. 
 
In the case $\alpha - \beta - \gamma \in \Phi$ we obtain from~\eqref{Chevalley-CCF1} that
$g_1 = x_{\beta - \alpha + \gamma} (N_{\gamma, \beta -\alpha}\cdot bf)     \cdot x_{\gamma}(f).$ Notice that $\langle \alpha - \beta - \gamma, \alpha \rangle = \langle \alpha, \alpha \rangle - \langle \beta, \alpha \rangle - \langle \gamma, \alpha \rangle \leq 2 - 1 - 1 = 0$, 
 therefore $\langle \alpha - \beta - \gamma, \alpha \rangle = -1,0$. Thus, $g_1$ is a product of generators of type P2, P2 or P3, P2.

Now suppose that the other alternative holds, namely that $\alpha - \beta + \gamma \in \Phi$.
Since $\langle \alpha - \beta + \gamma, \alpha \rangle = \langle \alpha , \alpha \rangle - \langle \beta, \alpha \rangle + \langle \gamma, \alpha \rangle \geq 2 - 1 + 1 = 2$
we obtain that $\alpha - \beta + \gamma = \alpha$, i.\,e. $\beta=\gamma$. Thus, by~\eqref{Chevalley-CCF2} $g_1 = x_\beta(f)^{x_{\beta-\alpha}(b)} = x_\beta(f)$. \end{proof}

\subsection{The map \texorpdfstring{$S_\alpha(a, -)$}{S(a,-)} and its properties} \label{sec:S-a}
Throughout this section $(A, M)$ is a local pair and, as before, $\alpha, \beta$ are some fixed roots of $\Phi$ forming an acute angle.

For $m \in M$ denote by $P_\alpha(m)$ the coset $P_\alpha(0) \cdot x_{-\alpha}(mX)$.
From~\cref{P0_normal} it follows that $P_\alpha(*)$ coincides with the union of all $P_\alpha(m)$, $m\in M$. 
\begin{df}\label{S-def}
Define the map $S_\alpha(a, -) \colon P_\alpha(*) \to \St(\Phi, A[X, X^{-1}])$ on each coset $P_\alpha(m)$ via the following formula:
\begin{equation}\label{eq:S-def} S_\alpha(a, g) = x_\alpha(aX^{-1})\cdot g \cdot x_\alpha\left(-\tfrac{aX^{-1}}{1 + am}\right) \cdot \{X, 1+ am\}.\end{equation} \end{df}
Notice that the restriction of the map $S_\alpha(a, -)$ to the subgroup $P_\alpha(0)$ coincides with the automorphism of left conjugation by $x_\alpha(aX^{-1})$.

In the sequel we often use the following property of $S_\alpha(a, -)$.
\begin{lemma} \label{lem:Smult}
For $g_1 \in P_\alpha(m)$, $g_2 \in P_\alpha(*)$ one has
\[S_\alpha(a, g_1\cdot g_2) = S_\alpha(a, g_1) \cdot S_\alpha\left(\tfrac{a}{1+am}, g_2\right).\]
\end{lemma} \begin{proof} Suppose $g_2\in P_\alpha(m')$ for some $m' \in M$, so that $g_1 \cdot g_2 \in P_\alpha(m + m')$. The assertion now follows from the definition of $S_\alpha(a, -)$ and~\eqref{eq:symbol-properties}:

\begin{multline*} S_\alpha(a, g_1 \cdot g_2) = x_\alpha(aX^{-1}) \cdot g_1 \cdot g_2 \cdot x_\alpha\left(-\tfrac{aX^{-1}}{1+am+am'}\right) \cdot \{X, 1 + am + am'\} = \\ = x_\alpha(aX^{-1}) g_1 x_\alpha\left(-\tfrac{aX^{-1}}{1+am}\right) \cdot x_\alpha\left(\tfrac{aX^{-1}}{1+am}\right) g_2 x_\alpha\left(-\frac{\tfrac{a}{1+am}\cdot X^{-1}}{1 + \tfrac{a}{1+am}\cdot m'}\right) \cdot \{X, (1+am)\left(1 + \tfrac{am'}{1+am}\right)\} = \\ = S_\alpha(a, g_1) \cdot S_\alpha\left(\tfrac{a}{1+am}, g_2\right). \qedhere \end{multline*}
\end{proof}

\begin{lemma} \label{lem:Tulenbaev-formula} 
One has $S_\alpha(a, x_{-\alpha}(mX)) = x_{-\alpha}\left(\tfrac{mX}{1+am}\right) \cdot \langle a, m\rangle_\alpha \cdot h_\alpha(1+am).$
\end{lemma}
\begin{proof} Since $\Phi$ is nonsymplectic, we can choose $\gamma \in \Phi$ such that $\langle \alpha, \gamma \rangle = -1$. Direct computation shows that
 \begin{align} 
 x_\alpha(aX^{-1}) \cdot x_{-\alpha}(mX) = {}^{h_\gamma(X)}(x_\alpha(a) \cdot x_{-\alpha}(m)) &\text{ by~\eqref{eq:conj-h-x}} \label{eq:Tf-comp} \\
 = {}^{h_\gamma(X)}\left( x_{-\alpha}\left(\tfrac{m}{1+am}\right) \cdot \langle a, m\rangle_\alpha \cdot h_\alpha(1+am) \cdot x_\alpha\left(\tfrac{a}{1+am}\right) \right) &\text{ by~\eqref{eq:dennis-stein}} \nonumber \\
 = x_{-\alpha}\left(\tfrac{mX}{1+am}\right) \cdot \langle a, m\rangle_\alpha \cdot h_\alpha(X^{-1}(1+am))\cdot h_\alpha^{-1}\left(X^{-1}\right) \cdot x_{\alpha}\left(\tfrac{aX^{-1}}{1+am}\right) &\text{ by~\eqref{eq:conj-h-x}, \eqref{eq:conj-h-h}} \nonumber \\
 = x_{-\alpha}\left(\tfrac{mX}{1+am}\right) \cdot \langle a, m\rangle_\alpha \cdot \{X^{-1}, 1+am\} \cdot h_\alpha(1+am)\cdot x_{\alpha}\left(\tfrac{aX^{-1}}{1+am}\right)&\text{ by~\eqref{eq:steinberg}} \nonumber \\
 = x_{-\alpha}\left(\tfrac{mX}{1+am}\right) \cdot \langle a, m\rangle_\alpha \cdot h_\alpha(1+am) \cdot x_{\alpha}\left(\tfrac{aX^{-1}}{1+am}\right) \cdot \{1+am, X\}&\text{ by~\eqref{eq:symbol-properties}, \eqref{eq:symbol-inverse}}. \nonumber \end{align}
The assertion of the lemma now follows from the following computation:
\begin{align*}
 S_\alpha(a, x_{-\alpha}(mX)) = x_\alpha(aX^{-1})\cdot x_{-\alpha}(mX) \cdot x_\alpha\left(-\tfrac{aX^{-1}}{1 + am}\right) \cdot \{X, 1+ am\} &\text{ by~\eqref{eq:S-def}} \\
 = x_{-\alpha}\left(\tfrac{mX}{1+am}\right) \cdot \langle a, m\rangle_\alpha \cdot h_\alpha(1+am) \cdot \{1+am, X\} \cdot \{X, 1+ am\}&\text{ by~\eqref{eq:Tf-comp}} \\
 = x_{-\alpha}\left(\tfrac{mX}{1+am}\right) \cdot \langle a, m\rangle_\alpha \cdot h_\alpha(1+am) &\text{ by~\eqref{eq:symbol-properties}}. \qedhere\end{align*} \end{proof}

\begin{corollary}\label{SR:additivity} For $g \in P_\alpha(m)$ one has $S_\alpha(a, g) \cdot h_\alpha^{-1}(1 + am) \in P(\alpha, \tfrac{m}{1 + am}).$
\end{corollary} \begin{proof}
Fix an element $g \in P_\alpha(m)$ and write it $g = g_0 \cdot x_{-\alpha}(mX)$ for some $g_0 \in P_\alpha(0)$.
Direct computation shows that
\begin{align*}
 S_\alpha(a, g) \cdot h_{\alpha}^{-1}(1+am) = S_\alpha(a, g_0) \cdot S_\alpha(a, x_{-\alpha}(mX)) \cdot h_{\alpha}^{-1}(1+am)&\text{ by~\cref{lem:Smult}} \\\
 = S_\alpha(a, g_0) \cdot x_{-\alpha}\left(\tfrac{mX}{1+am}\right) \cdot\langle a, m \rangle &\text{ by~\cref{lem:Tulenbaev-formula}.}
\end{align*}
By \cref{P0_conj} and \cref{rem:DS} the subgroup $P_\alpha(0)$ contains $S_\alpha(a, g_0)$ and the symbol $\langle a, m \rangle$. Thus, the expression in the right-hand side of the above formula lies in $P_\alpha\left(\tfrac{m}{1+am}\right)$. 
\end{proof}
\begin{lemma} \label{SR:sharp} For $g \in j_+(K(A[X], M[X])) \cap P_\alpha(m) \cap P_\beta(m')$ and $a \in A$ one has \[S_\alpha(a, g)\cdot  h^{-1}_\alpha(1 + am) \cdot x_{\alpha-\beta}(-N_{\alpha,-\beta}\cdot am') \in P_\beta\left(\tfrac{m'}{1 + am}\right). \] \end{lemma}
\begin{proof} By~\cref{K-a-b} $g$ can be presented as $g_0 \cdot x_{-\alpha}(mX) \cdot x_{-\beta}(m'X)$ for some $g_0\in K(\alpha, \beta)$. By~\cref{K-a-b-P-b} the element $g_1 = S_\alpha(a, g_0) $ lies in $P_\beta(0)$. Since $x_{-\beta}(m'X) \in P_\alpha(0)$ we get that
\newline\scalebox{.98}{\begin{minipage}{1.02\linewidth}\begin{align*}
  S_\alpha(a, g) = g_1 \cdot S_\alpha(a, x_{-\alpha}(mX)) \cdot S_\alpha\left(\tfrac{a}{1+am}, x_{-\beta}(m'X)\right) & \text{ by~\cref{lem:Smult}} \\ 
  = g_1 \cdot x_{-\alpha}\left(\tfrac{mX}{1+am}\right) \cdot \langle a, m \rangle \cdot h_\alpha(1+am) \cdot x_{-\beta}(m'X) \cdot x_{\alpha-\beta}\left(\tfrac{N_{\alpha, -\beta} \cdot am'}{1+am}\right) & \text{ by~\cref{lem:Tulenbaev-formula},\eqref{Chevalley-CCF1}} \\
  = g_1 \cdot x_{-\alpha}\left(\tfrac{mX}{1+am}\right) \cdot \langle a, m \rangle \cdot x_{-\beta}\left(\tfrac{m'X}{1+am}\right) \cdot x_{\alpha-\beta}\left(N_{\alpha, -\beta} \cdot am'\right) \cdot h_\alpha(1+am) & \text{ by~\eqref{eq:conj-h-x}}.
\end{align*}\end{minipage}} \vskip \belowdisplayshortskip
The required assertion now follows from~\cref{rem:palpha}.
\end{proof}

\begin{lemma} \label{SR:obtuse} For $g \in j_+(K(A[X], M[X])) \cap P_\alpha(m) \cap P_\beta(m')$ and $a, b\in A$ the element \[ x_{\beta - \alpha}(b) \cdot S_\alpha(a, g) \cdot h_{\alpha}^{-1}(1+am)\cdot h_{\beta - \alpha}((1 + \epsilon abm')^{-1})\cdot x_{\beta - \alpha}(-b(1 + \epsilon abm')) \]
 belongs to $P_\alpha\left(\tfrac{m - \epsilon bm'}{1+am}\right)$, where $\epsilon = N_{\alpha, -\beta}$.
\end{lemma}
\begin{proof} Set $c = 1 + \epsilon a b m'$, $s_0 = \langle a, m \rangle$, $s_1 = \langle \epsilon am', -bc^{-1} \rangle$.
Notice that $1 - \epsilon abc^{-1}m' = c^{-1}$, therefore by~\eqref{eq:dennis-stein}
\begin{equation} \label{eq:computation-SR-obtuse} x_{\alpha-\beta}(\epsilon am') \cdot x_{\beta-\alpha}(-bc^{-1}) = x_{\beta-\alpha}(-b) \cdot s_1 \cdot h_{\alpha-\beta}(c^{-1}) \cdot x_{\alpha-\beta}(\epsilon acm').\end{equation}
Notice that $x_{-\beta}(m'X) \in P_\alpha(0)$, therefore from the definition of $S_\alpha(a, -)$ we obtain that
\newline \scalebox{0.92}{\begin{minipage}{1.08\linewidth}
\begin{align*}
 S_\alpha\left(a, x_{-\beta}(m'X) \cdot x_{-\alpha}(mX)\right) \cdot h_{\alpha}^{-1}(1+am)\cdot h_{\beta-\alpha}(c^{-1})\cdot x_{\beta - \alpha}(-bc) \\ 
 = S_\alpha(a, x_{-\beta}(m'X)) \cdot S_\alpha(a, x_{-\alpha}(mX)) \cdot h_\alpha^{-1} (1 + am) \cdot h_{\beta - \alpha}(c^{-1}) \cdot x_{\beta - \alpha}(-bc) &\text{ by~\cref{lem:Smult}}\\
 = S_\alpha(a, x_{-\beta}(m'X)) \cdot x_{-\alpha}\left(\tfrac{mX}{1+am}\right) \cdot s_0 \cdot h_{\beta - \alpha}(c^{-1}) \cdot x_{\beta - \alpha}(-bc) &\text{ by~\cref{lem:Tulenbaev-formula}}\\
 = x_{-\beta}(m'X) \cdot x_{\alpha - \beta}(\epsilon am') \cdot x_{\beta - \alpha}(-bc^{-1}) \cdot x_{-\alpha}\left(\tfrac{mX}{1+am}\right) \cdot s_0 \cdot h_{\beta - \alpha}(c^{-1}) &\text{ by~\eqref{Chevalley-CCF1},~\eqref{Chevalley-CCF2}, \eqref{eq:conj-h-x}}\\
 \begin{split}= x_{-\beta}(m'X) \cdot x_{\beta - \alpha}(-b) \cdot s_1 \cdot h_{\alpha-\beta}(c^{-1}) \cdot \hspace{100pt} \\ \cdot x_{\alpha-\beta}\left(\epsilon acm'\right) \cdot x_{-\alpha}\left(\tfrac{mX}{1+am}\right) \cdot s_0 \cdot h_{\beta - \alpha}(c^{-1}) &\text{ by~\eqref{eq:computation-SR-obtuse}} \end{split}\\
 =  x_{\beta - \alpha}(-b) \cdot x_{-\beta}(m'X) \cdot x_{-\alpha}(-\epsilon bm'X) \cdot s_1 \cdot x_{\alpha-\beta}(\epsilon ac^{-1}m') \cdot x_{-\alpha}\left(\tfrac{cmX}{1+am}\right) \cdot s_0 &\text{ by~\eqref{Chevalley-CCF1},~\eqref{eq:conj-h-x},~\eqref{eq:h-inv}}\\
 = x_{\beta-\alpha}(-b) \cdot x_{-\beta}(m'c^{-1}X) \cdot x_{\alpha-\beta}(\epsilon ac^{-1}m') \cdot s_1 \cdot s_0 \cdot x_{-\alpha}\left(-\epsilon bm' X + \tfrac{cmX}{1+am}\right) &\text{ by~\eqref{Chevalley-CCF1}.}
\end{align*}\end{minipage}} \vskip \belowdisplayshortskip
Denote by $h_0$ the product of all factors in the above formula except the first and the last one. It is clear that $h_0$ lies in $P_\alpha(0)$. By~\cref{K-a-b} we can decompose $g$ as $g_0 \cdot x_{-\beta}(m'X) \cdot x_{-\alpha}(mX)$ for some $g_0 \in K(\alpha, \beta)$, $m, m' \in M$. The required assertion now follows from Lemmas~\ref{conj-S-a-x-bma}, \ref{lem:Smult} and the above computation:
\begin{multline} \nonumber
 x_{\beta - \alpha}(b) \cdot S_\alpha(a, g) \cdot h_{\alpha}^{-1}(1+am)\cdot h_{\beta - \alpha}(c^{-1})\cdot x_{\beta - \alpha}(-bc) = \\
  = x_{\beta - \alpha}(b) \cdot S_\alpha(a, g_0) \cdot x_{\beta-\alpha}(-b) \cdot h_0 \cdot x_{-\alpha}\left(\tfrac{(m -\epsilon bm')X}{1+am}\right) \in P_\alpha(\tfrac{m -\epsilon bm'}{1+am}). \qedhere
\end{multline}
\end{proof}

\subsection{Construction of a \texorpdfstring{$\St(\Phi, B)$}{St(B)}-torsor} \label{sec:V-construction}
Throughout this subsection $\Phi$ denotes an arbitrary irreducible simply-laced root system of rank $\geq 3$, unless stated otherwise. 
\begin{df}\label{df:GM0_geq0}
Denote the subgroup $\overline{\St}(\Phi, A, M)$ by $G_M^0$. 
Consider the following commutative diagram, in which the homomorphisms $i_\pm$ and $j_\pm$ are induced by the natural ring homomorphisms $A \to A[X^{\pm 1}]$ and $A[X^{\pm 1}]\to R$:
\[\begin{tikzcd}[scale cd=0.8]
   G_M^0 \arrow[rr, "i_+^M", hookrightarrow] \arrow[dd, "i_-^M", hookrightarrow] \arrow[rd, hookrightarrow] & & \overline{\St}(\Phi, A[X], M[X]) \arrow[dd, "j_+^M", swap, near start] \arrow[rd, hookrightarrow] & \\ & \St(\Phi, A) \arrow[rr, "i_+", near start, hookrightarrow] \arrow[dd, "i_-", near start, hookrightarrow] & & \St(\Phi, A[X]) \arrow{dd}{j_+} \\ \overline{\St}(\Phi, A[X\inv], M[X\inv]) \arrow[rd, hookrightarrow] \arrow[rr, "j_-^M", near start] & & \overline{\St}(\Phi, R, I) \arrow[rd, hookrightarrow] & \\ & \St(\Phi, A[X\inv]) \arrow{rr}[swap]{j_-} & & \St(\Phi, R).
  \end{tikzcd}\]
Denote by $\overline{G}^{\geq 0}_M$ the image of the homomorphism $j_+^M$. \end{df}
Notice that the homomorphism $j_+^Mi_+^M$ is split by the homomorphism $ev^*_{X=1}$ of evaluation at $1$, therefore $G_M^0$ can be considered as a subgroup of $\overline{G}^{\geq 0}_M$.

Denote by $\overline{V}_T$ the quotient of the set of triples 
\begin{equation}\label{VT-def} V_T = \overline{G}_M^{\geq 0} \times \St(\Phi, A[X\inv]) \times (1+M)^\times \end{equation} by the equivalence relation given by $(p \cdot j_+i_+(\gamma), h, u)_T \sim (p, i_-(\gamma) \cdot h, u)_T$, $\gamma \in G^0_M$. We denote the image of $(p, h, u)_T$ in $\overline{V}_T$ by $[p, h, u]$.
$\overline{V}_T$ is precisely the set upon which M.~Tulenbaev in~\cite[Proposition~4.3]{Tu83} constructs an action of $\St(\Phi, B)$.

Sometimes it will be more convenient for us to work with another set $\overline{V}$ isomorphic to $\overline{V}_T$ (this isomorphism will be established below). Denote by $V$ the subset of \[\St(\Phi, R) \times \St(\Phi, A[X\inv]) \times (1 + M)^\times\] consisting of those triples $(g, h, u)$ for which $p(g, h, u) := g \cdot j_-(h) \cdot \{ X, u \}$ belongs to $\overline{G}_M^{\geq 0}$. 

We let $h_0 \in G_M^0$ act on $V$ on the right by $(g, h, u) \cdot h_0 = (g, h \cdot i_-(h_0), u)$. Since $G^0_M \subset \overline{G}^{\geq 0}_M$, we see that $V \cdot G_M^0 \subseteq V$.

We denote by $\overline{V}$ the set of orbits of this action and use the notation $(g, [h], u)$ for the elements of $\overline{V}$.
Whenever $v_1, v_2 \in V$ lie in the same $G_M^0$-orbit we use the notation $v_1 \sim v_2$.
We denote by $\overline{p}$ the function $\overline{V} \to \overline{G}^{\geq 0}_M/G_M^0$ sending each $(g, [h], u) \in V$ to the left coset $p(g, h, u)G_M^0$.

The isomorphism between the sets $\overline{V}$ and $\overline{V}_T$ is given by the following two maps, which are easily seen to be mutually inverse to each other:
\begin{equation} \label{eq:VVT} \begin{tikzcd}[row sep=tiny] \overline{V} \arrow{r}{\cong} & \arrow{l} \overline{V}_T \\ (g, [h], u) \arrow[r, mapsto] & \left[p(g, h, u), h^{-1}, u\right] \\ (p \cdot j_-(h) \cdot \{u, X\}, [h^{-1}], u) & \arrow[l, mapsto] [p, h, u]. \end{tikzcd} \end{equation}
The above isomorphism allows us to regard $\overline{V}$ and $\overline{V}_T$ as the same object, for elements of which we can interchangeably use either of the two notations,
 depending on which of them is more convenient in a given situation.
For example, specifying the action of $\St(\Phi, B)$ in terms of $\overline{V}$ leads to much shorter calculations in Lemmas~\ref{R3_leq0_1}--\ref{R2_0_1},
 while the statements of~\cref{prop43} and~\cref{lem:action} look more natural when formulated in terms of~$\overline{V}_T$.

Now we are ready to proceed with the construction of the action of $\St(\Phi, B)$ on $\overline{V}$. We start by defining for $\alpha \in \Phi$, $a \in A$ a partial function $t_\alpha(aX^{-1}) \colon V \not\to V$.
This function is defined for the triples $(g, h, u)$ satisfying $p(g, h, u) \in P_\alpha(*) \subseteq \overline{G}_M^{\geq 0}$.
If $p(g, h, u)$ belongs to $P_\alpha(m)$ for some $m \in M$, then $t_\alpha(aX^{-1})$ is defined via the following identity:
\begin{equation} \label{T_1} t_\alpha(aX^{-1}) (g, h, u) = \left( x_\alpha(aX^{-1})\cdot g ,\ h \cdot x_\alpha\left(-\tfrac{aX^{-1}}{1 + am}\right),\ u \cdot (1 + am)\right).\end{equation}

Notice that $t_\alpha(aX^{-1}) (g, h, u) \in V$. Indeed, from~\cref{SR:additivity} we obtain that
\[p\left(t_\alpha(aX^{-1}) (g, h, u)\right) = S_\alpha(a, p(g, h, u)) \in P_\alpha(*) \cdot h_{\alpha}^{-1}(1+am) \subseteq \overline{G}_M^{\geq 0}.\]
\begin{lemma}\label{lem:orbit-action} Let $(A, M)$ be a local pair. 
Then for any $\alpha \in \Phi$ and $a \in A$ the partial function $t_\alpha(aX^{-1}) \colon V \not \to V$ gives rise to a well-defined total function $T_\alpha(aX^{-1}) \colon \overline{V} \to \overline{V}$. \end{lemma}
\begin{proof}
 First of all, let us show that the resulting function is total. 
 Fix $v_0 = (g, h, u) \in V$. Since $p(g, h, u) \in \overline{G}^{\geq 0}_M$ there exists $g_1 \in \overline{\St}(\Phi, A[X], M[X])$ such that $j_+(g_1) = p(g, h, u).$ Notice that the homomorphism $i_+$ is split by $ev^*_{X=0}$.
 Set $h_0 = ev^*_{X=0}(g_1)^{-1}$, then, clearly, $g_1 \cdot i_+(h_0) \in K(A[X], M[X])$ and 
 $p(g, h \cdot i_-(h_0), u) = j_+(g_1) \cdot j_-(i_-(h_0)) = j_+(g_1 \cdot i_+(h_0)) \in j_+(K(A[X], M[X])).$ The latter subgroup is contained in $P_\alpha(*)$ by~\cref{Pstar-large}. Thus, $t_\alpha(aX^{-1})$ is defined on the representative $(g, h \cdot i_-(h_0), u)$ lying in the same $G_M^0$-orbit as $v_0$.
  
 Next, let us show that the value of $T_\alpha(aX^{-1})$ does not depend on the choice of representative.
 Let $v_1 = (g, h_1, u)$ and $v_2 = (g, h_2, u)$ be two elements of the same $G_M^0$-orbit for which both $p(v_1)$ and $p(v_2)$ belong to $P_\alpha(*)$.
 By definition, $h_1^{-1} h_2 = i_-(h_0)$, for some $h_0 \in G^0_M$, moreover, 
  $p(v_1)^{-1} \cdot p(v_2) = j_- i_-(h_0) = j_+ i_+(h_0) \in P_\alpha(*)$.
 By~\cref{Pstar-char} there exists $g_1 \in \overline{\St}(\Phi, A[X], M[X])$ such that $g_0 := ev^*_{X=0}(g_1) \in Z_\alpha(A, M)$ and $j_+(g_1) = j_+ i_+ (h_0)$.
 From the last equality and the injectivity of the homomorphism $\GG(\Phi, A[X]) \to \GG(\Phi, R)$ we obtain that the projections of $g_0$, $g_1$ and $h_0$ in $\GG(\Phi, R)$ are equal, which shows that $g_0 \cdot h_0^{-1} \in \overline{\K_2}(\Phi, A, M)$. It follows from~\cref{thm:Stein} that the latter subgroup is generated by relative Steinberg symbols $\{ a, 1 + m \}$ and hence by~\cref{Z-DS} it is contained in $Z_\alpha(A, M)$. Thus, we have obtained that $h_0 \in Z_\alpha(A, M)$ and hence that $i_-(h_0)$ is centralized by $X_\alpha(A[X^{-1}])$, which allows us to conclude that
 \begin{multline} \nonumber
  t_\alpha(aX^{-1})(v_1) = \left( x_\alpha(aX^{-1})\cdot g ,\ h_1 \cdot x_\alpha\left(-\tfrac{aX^{-1}}{1 + am}\right),\ u \cdot (1 + am)\right) \sim \\
  \sim \left( x_\alpha(aX^{-1})\cdot g ,\ h_1 i_-(h_0) \cdot x_\alpha\left(-\tfrac{aX^{-1}}{1 + am}\right),\ u \cdot (1 + am)\right) = t_\alpha(aX^{-1})(v_2). \qedhere
 \end{multline}\end{proof}

Our next goal is to define for $a + Xf \in A + XM[X]$ the operator $T_\alpha(a + Xf) \colon \overline{V} \to \overline{V}$.
Let $(g, h, u)$ be an element of $V$. We define the value of $T_\alpha(a + Xf)$ on $(g, h, u)$ via the following identity:
\begin{equation} \label{T_leq0} T_\alpha(a + Xf) \cdot (g, h, u) = (x_\alpha(a + Xf) \cdot g, h \cdot x_{\alpha}(-a), u).  \end{equation}
Notice that $p(T_\alpha(a + Xf) \cdot (g, h, u)) = x_\alpha(Xf) \cdot {}^{x_\alpha(a)}p(g, h, u) \in \overline{G}_M^{\geq 0},$ therefore
 the right-hand side of~\eqref{T_leq0} lies in $V$. Notice that for $h_0 \in G^0_M$ one has
\[T_\alpha(a + Xf) \cdot ((g, h, u) \cdot h_0) = (T_\alpha(a + Xf) \cdot (g, h, u)) \cdot {}^{x_{\alpha}(a)}\!h_0,\]  
therefore~\eqref{T_leq0} indeed gives rise to a well-defined map $\overline{V} \to \overline{V}$.

Thus far we have specified the action of the generators of the "truncated'' Steinberg group $\St^{\leq 1}(\Phi, B)$ from~\cref{sec:presentation} on $\overline{V}$ using formulae~\eqref{T_1} and~\eqref{T_leq0}. We need to verify that this action respects the defining relations of the group $\St^{\leq 1}(\Phi, B)$ (notice that $\leq 1$ here stands for the degree of relations with respect to $t = X^{-1}$). This is accomplished in the series of lemmas below.

\begin{lemma} \label{R1_d} The operators $T_\alpha$ satisfy Steinberg relations of type $\mathrm{R1}_d$ for all $d \leq 1$. \end{lemma}
\begin{proof}
 For $d \leq 0$ the assertion immediately follows from~\eqref{T_leq0}, so let us consider the case $d = 1$. 
 
 Let $a$ and $b$ be elements of $A$ and $(g, h, u)$ be an element of $V$.
 By the first part of the proof of~\cref{lem:orbit-action} we may assume that $p(g,h,u)\in P_\alpha(m)$ for some $m \in M$.
 From~\eqref{T_1} we obtain that 
 \begin{equation}\label{add-computation}
  T_\alpha(bX^{-1}) \cdot T_\alpha(aX^{-1}) (g, [h], u) = T_\alpha(bX^{-1}) \cdot \left(x_\alpha(aX^{-1})\cdot g ,\ \left[ h' \right],\ u \cdot (1 + am)\right),
 \end{equation}
 where $h' = h \cdot x_\alpha\left(-\tfrac{aX^{-1}}{1 + am}\right)$.
 Notice that we can not invoke~\eqref{T_1} for the second time because the value of the function $p$ on the triple $(x_\alpha(aX^{-1})\cdot g ,\ h',\ u \cdot (1 + am))$ does not belong to $P_\alpha(*)$.
 
 However, it is very easy to fix this. Indeed, we are free to replace $[h]$ with $[h \cdot h_0]$ for any $h_0 \in G_M^0$, so in~\eqref{add-computation} we
  can replace $h'$ with $h'' = h' \cdot h^{-1}_\alpha(1+am)$. Now from~\cref{SR:additivity} we obtain that
 \[ p(x_\alpha(aX^{-1})\cdot g ,\ h'',\ u \cdot (1 + am)) = S_\alpha(a, p(g, h, u)) \cdot h^{-1}_\alpha(1+am) \in P_\alpha\left(\tfrac{m}{1+am}\right).\]
 Now we can invoke~\eqref{T_1} once again and can continue~\eqref{add-computation} as follows:
 \begin{equation}\label{add-computation2}
  \ldots = \left(x_\alpha((a + b)X^{-1})\cdot g ,\ \left[ h'' \cdot x_\alpha\left(\tfrac{-bX^{-1}}{1+\tfrac{bm}{1+am}}\right) \right],\ u \cdot (1 + am) \cdot \left(1+\tfrac{bm}{1+am}\right)\right).\end{equation}
 Notice that 
 \begin{multline} \nonumber \left[h'' \cdot x_\alpha\left(\tfrac{-b(1+am)X^{-1}}{1+am + bm}\right)\right] = \left[h \cdot x_\alpha\left(\tfrac{-aX^{-1}}{1 + am}\right) \cdot x_\alpha\left(\tfrac{-bX^{-1}}{(1+am)(1+am+bm)}\right)\right] = \\ =
 \left[h \cdot x_\alpha\left(\tfrac{-(a+b)X^{-1}}{1+(a+b)m}\right)\right].\end{multline}
 Thus, the expression in~\eqref{add-computation2} coincides with the expression for $T_\alpha((a+b)X^{-1})\cdot (g, [h], u)$ given by~\eqref{T_1}. \end{proof}

\begin{lemma} \label{R3_leq0_1} The operators $T_\alpha$ satisfy Steinberg relations of type $\mathrm{R3}^\angle_{d, 1}$, $\mathrm{R3}^\bot_{d,1}$ for $d\leq 0$. \end{lemma}
\begin{proof}
We will verify $\mathrm{R3}^\angle_{d, 1}$ and $\mathrm{R3}^\bot_{d, 1}$ simultaneously. Fix some $\beta \in Z(\alpha) = Z_0(\alpha) \sqcup Z_+(\alpha)$.

We need to show that $[T_\beta(b), T_\alpha(aX^{-1})](g, [h], u) = (g, [h], u)$ for $a\in A$, $b\in A + XM[X]$. 
Write $b = b_0 + Xf$ for some $b_0 \in A$, $f \in M[X]$.
As in the proof of the previous lemma, we may assume that $p(g, h, u) \in P_\alpha(m)$ for some $m \in M$.
From~\eqref{T_leq0}--\eqref{T_1} we obtain that
\begin{equation} \label{R3leq0-computation}
  \left[T_\beta(b),\ T_\alpha(aX^{-1}) \right] (g, [h], u) = T_\beta(b) \cdot T_\alpha(aX^{-1}) \left(g', [h'], u(1-am)\right), \end{equation}
where $g' = x_\beta(-b) \cdot x_\alpha(-aX^{-1}) \cdot g$, $h' = h \cdot x_{\alpha}\left(\tfrac{aX^{-1}}{1-am}\right) \cdot x_\beta(b_0)$.
Notice that by~\eqref{eq:conj-h-x} $[h'] = [h'']$, where $h'' = h \cdot x_{\alpha}\left(\tfrac{aX^{-1}}{1-am}\right) \cdot h^{-1}_\alpha(1-am) \cdot x_\beta(b_0).$
From~\cref{P0-conj-invariant} and \cref{SR:additivity} we obtain that
\begin{multline*} p(g',h'',u(1-am)) = x_{\beta}(-Xf) \cdot \left(S_\alpha(-a, p(g,h,u)) \cdot h_{\alpha}^{-1}(1-am)\right)^{x_\beta(b_0)} \in \\ \in x_{\beta}(-Xf) \cdot P_\alpha(\tfrac{m}{1-am})^{x_\beta(b_0)} \subseteq P_\alpha(\tfrac{m}{1-am}).\end{multline*}
Thus, we can invoke ~\eqref{T_1}--\eqref{T_leq0} once again and can continue~\eqref{R3leq0-computation} as follows:
\begin{equation} \label{R3leq0-computation2} \ldots = \left( [x_\beta(b), x_{\alpha}(aX^{-1})] \cdot g, [h'' \cdot x_{\alpha}\left(-a(1-am)X^{-1}\right) \cdot x_{\beta}(-b_0)], u \right) = \left(g, [h], u\right), \end{equation}
where in the last equality we use the following computation:
\begin{multline} \nonumber
 h'' \cdot x_{\alpha}\left(-a(1-am)X^{-1}\right) \cdot x_{\beta}(-b_0) = \\ = h \cdot h_\alpha^{-1}(1-am) \cdot \left[x_\alpha\left(a(1-am)X^{-1}\right),\ x_\beta(b_0)\right] = h \cdot h_\alpha^{-1}(1-am). \end{multline}
The assertion of the lemma now follows from~\eqref{R3leq0-computation}--\eqref{R3leq0-computation2}. \end{proof}

\begin{lemma} \label{R3_leqm1_1} The operators $T_\alpha$ satisfy Steinberg relations of type $\mathrm{R2}_{d, 1}$ for $d \leq -1$ . \end{lemma}
\begin{proof} Let $a \in A$, $f \in M[X]$ and $\alpha$, $\beta$ be a pair of roots forming an obtuse angle. As before, we may assume that $p(g, h, u) \in P_\alpha(m)$ for some $m \in M$.
From~\eqref{T_1}--\eqref{T_leq0} and the fact that $x_\beta(Xf) \in P_\alpha(0)$ we obtain that
\newline\scalebox{0.97}{\begin{minipage}{1.03\linewidth}
\begin{multline} \nonumber [T_\beta(Xf),\ T_\alpha(aX^{-1})] \left(g,\ [h],\ u\right) = \\ 
= T_\beta(Xf) \cdot T_\alpha(aX^{-1}) \left(x_{\beta}(-Xf) \cdot x_\alpha(-aX^{-1})\cdot g,\ \left[h \cdot x_\alpha\left(\tfrac{aX^{-1}}{1 - am}\right) \cdot h^{-1}_\alpha(1-am)\right],\ u(1-am)\right) = \\ = \left([x_\beta(Xf),\ x_\alpha(aX^{-1})] \cdot g,\ \left[h \cdot x_\alpha\left(-\tfrac{aX^{-1}}{1-am}\right) \cdot h^{-1}_\alpha(1-am) \cdot x_\alpha\left(\tfrac{aX^{-1}}{1 + \tfrac{am}{1-am}}\right)\right],\ u \right) = \\ = (x_{\alpha+\beta}(N_{\beta,\alpha} \cdot af) \cdot g,\ [h],\ u) = T_{\alpha+\beta}(N_{\beta, \alpha} \cdot af) (g, [h], u). \end{multline}\end{minipage}}\vskip \belowdisplayshortskip

As in the proof of Lemmas~\ref{R1_d}--\ref{R3_leq0_1}, in the above computation we had to add the factor $h_\alpha^{-1}(1-am)$ to the second triple, so that we could pass from the second line to the third one (again, we need to invoke \cref{SR:additivity} to check that the value of $p$ on the second triple lies in the coset $P_\alpha(\tfrac{m}{1-am})$). \end{proof}

\begin{lemma} \label{R3_1_1} The operators $T_\alpha(aX^{-1})$ satisfy Steinberg relations $\mathrm{R3}_{1,1}^{\angle}$. \end{lemma}
\begin{proof}
Let $a, b \in A$ and $\alpha$, $\beta$ be a pair of roots forming an acute angle. Set $\epsilon = N_{\alpha,-\beta}$. From~\cref{K-a-b} and the proof of the first part of~\cref{lem:orbit-action} it follows that \[p(g, h, u) \in j_+\left(K(A[X], M[X])\right) \cap P_\alpha(m) \cap P_{\beta}(m')\text{ for some }m,m' \in M.\]
From~\eqref{T_1} we obtain that
\begin{multline} \label{R3_1_1-computation}
 T_\beta(bX^{-1}) \cdot T_\alpha(aX^{-1}) \left(g, [h], u\right) = \\
 = T_\beta(bX^{-1}) \left(x_\alpha(aX^{-1})\cdot g,\ \left[h \cdot x_\alpha\left(-\tfrac{aX^{-1}}{1 + am}\right)\right],\ u(1 + am)\right) = \\ 
 = T_\beta(bX^{-1}) \left(x_\alpha(aX^{-1})\cdot g,\ \left[h'\right],\ u(1 + am)\right),
 \end{multline}
where $h' = h \cdot x_\alpha\left(-\tfrac{aX^{-1}}{1 + am}\right) \cdot h^{-1}_\alpha(1 + am) \cdot x_{\alpha-\beta}(-\epsilon am').$
Notice that by~\cref{SR:sharp}
\[ p\left(x_\alpha(aX^{-1})\cdot g, h', u(1 + am)\right) = S_\alpha(a, p(g, h, u)) \cdot h^{-1}_\alpha(1 + am) \cdot x_{\alpha-\beta}(-\epsilon am') \in P_\beta(\tfrac{m'}{1+am}), \]
therefore we can invoke~\eqref{T_1} once again and can continue~\eqref{R3_1_1-computation} as follows:
\begin{equation}\label{R3_1_1-computation2} \ldots = \left(x_\beta(bX^{-1}) \cdot x_\alpha(aX^{-1}) \cdot g,\ \left[h' \cdot x_\beta\left(-\tfrac{bX^{-1}}{1 + \tfrac{bm'}{1+am}}\right)\right],\ u(1 + am + bm')\right). \end{equation}
The second component of the above triple can be simplified using~\eqref{eq:conj-h-x} as follows:
\begin{multline*}
 \left[ h \cdot x_\alpha\left(-\tfrac{aX^{-1}}{1 + am}\right) \cdot h^{-1}_\alpha(1 + am) \cdot x_{\alpha-\beta}(-\epsilon am') \cdot x_\beta\left(-\tfrac{bX^{-1}(1+am)}{1 + am + bm'}\right)\right] = \\
 = \left[ h \cdot x_\alpha\left(-\tfrac{aX^{-1}}{1 + am}\right) \cdot h^{-1}_\alpha(1 + am) \cdot x_{\alpha}\left(\tfrac{m'abX^{-1}(1+am)}{1 + am + bm'}\right) \cdot x_\beta\left(-\tfrac{bX^{-1}(1+am)}{1 + am + bm'}\right)\right] = \\
 = \left[ h \cdot x_\alpha\left(-\tfrac{aX^{-1}}{1 + am} + \tfrac{m'abX^{-1}}{(1+am)(1 + am + bm')}\right) \cdot x_\beta\left(-\tfrac{bX^{-1}}{1 + am + bm'}\right) \cdot h^{-1}_\alpha(1 + am)\right] = \\
 = \left[ h \cdot x_\alpha\left(-\tfrac{aX^{-1}}{1 + am + bm'}\right) \cdot x_\beta\left(-\tfrac{bX^{-1}}{1 + am + bm'}\right)\right].
\end{multline*}
Thus, we see that the expression in the right-hand side of~\eqref{R3_1_1-computation2} would remain unchanged if we swapped $(a,\alpha,m)$ with $(b,\beta,m')$. This implies the required assertion.
\end{proof}

\begin{lemma} \label{R2_0_1} The operators $T_\alpha$ satisfy Steinberg relations $\mathrm{R2}_{0,1}$. \end{lemma}
\begin{proof} 
Let $a, b \in A$ and $\alpha, \beta$ be a pair of roots such that $\alpha + \beta \in \Phi$. Set $\epsilon = N_{\alpha, \beta},$ $c = 1 - \epsilon abm'$.

As before, we may assume that \[p(g, h, u) \in j_+\left(K(A[X], M[X])\right) \cap P_\alpha(m) \cap P_{\alpha + \beta}(m')\text{ for some $m, m' \in M$.} \]

From~\eqref{T_1}--\eqref{T_leq0} and~\eqref{eq:conj-h-x} we obtain that
\begin{multline} \label{R2_0_1-computation}
[T_\beta(b),\ T_\alpha(aX^{-1})] \left(g, [h], u \right) = \\
= T_\beta(b) \cdot T_\alpha(aX^{-1}) \left(x_\beta(-b) \cdot x_\alpha (-a X^{-1})\cdot  g,\ [h\cdot x_\alpha\left(\tfrac{a X^{-1}}{1-am}\right) \cdot x_{\beta}(b)],\ u (1 - am) \right) = \\
= T_\beta(b) \cdot T_\alpha(aX^{-1}) \left(x_\beta(-b) \cdot x_\alpha (-a X^{-1})\cdot  g,\ [h'],\ u (1 - am) \right),
 \end{multline}
 where $h' = h\cdot x_\alpha\left(\tfrac{a X^{-1}}{1-am}\right) \cdot h_{\alpha}^{-1}(1-am)\cdot h_{\beta}(c^{-1}) \cdot x_{\beta}(bc)$. Applying~\cref{SR:obtuse} to the pair of roots $\alpha, \alpha+\beta$ we obtain that
\begin{multline*} p\left(x_\beta(-b) \cdot x_\alpha (-a X^{-1})\cdot  g,\ h',\ u (1 - am) \right) = \\ =
 x_\beta(-b) \cdot S_\alpha(-a, p(g,h,u)) \cdot h_{\alpha}^{-1}(1-am)\cdot h_{\beta}(c^{-1}) \cdot x_{\beta}(bc) \in P_\alpha\left(\tfrac{m-\epsilon bm'}{1-am}\right). \end{multline*}
Thus, we can continue~\eqref{R2_0_1-computation} using~\eqref{T_1}--\eqref{T_leq0}:
\begin{equation} \label{R2_0_1-computation2} \ldots = \left([x_\beta(b),\ x_\alpha(aX^{-1})]\cdot g,\ [h''],\ uc\right) = T_{\alpha+\beta}(-\epsilon abX^{-1})(g,\ [h],\ u), \end{equation} where the last equality is obtained from~\eqref{eq:conj-h-x} as follows:
\begin{multline} \nonumber
 [h''] = \left[ h' \cdot x_\alpha\left(\tfrac{-aX^{-1}}{1+\tfrac{a(m-\epsilon bm')}{1-am}}\right) \cdot x_\beta(-b) \right] = \\ = \left[ h\cdot x_\alpha\left(\tfrac{a X^{-1}}{1-am}\right) \cdot h_{\alpha}^{-1}(1-am)\cdot h_{\beta}(c^{-1})\cdot x_{\alpha + \beta}\left(\epsilon ab(1-am)X^{-1}\right)\cdot x_\alpha\left(-\tfrac{a(1 - am)X^{-1}}{c}\right) \right] = \\
 = \left[ h\cdot x_\alpha\left(\tfrac{a X^{-1}}{1-am}\right) \cdot x_{\alpha + \beta}\left(\tfrac{\epsilon \cdot abX^{-1}}{c}\right)\cdot x_\alpha\left(-\tfrac{aX^{-1}}{1 - am}\right) \cdot h_{\alpha}^{-1}(1-am)\cdot h_{\beta}(c^{-1}) \right]   = \\ = \left[ h\cdot x_{\alpha + \beta}\left(\tfrac{\epsilon   abX^{-1}}{1 - \epsilon a b m'}\right) \right].\end{multline}
The assertion of the lemma now follows from~\eqref{R2_0_1-computation}--\eqref{R2_0_1-computation2}.
\end{proof}

Now we are ready to prove the main result of this subsection.
\begin{prop} \label{prop43}
For $\Phi$ as in the statement of~\cref{lemma33} the operators $T_\alpha$ defined above specify a well-defined action of $\St(\Phi, B)$ on $\overline{V}$. 

This action satisfies the following additional properties.
\begin{enumerate}
 \item For any $h_1 \in \St(\Phi, A[X\inv])$ one has $j_B^-(h_1) \cdot [1, h, u] = [1, h_1 h, u]$, where $j_B^-$ denotes the homomorphism $\St(\Phi, A[X\inv]) \to \St(\Phi, B)$ (we identify $\overline{V}$ with $\overline{V}_T$ using the isomorphism~\eqref{eq:VVT}).
 \item If we consider $\St(\Phi, A[X, X\inv])$ as a set with the left multiplication action of $\St(\Phi, B)$ then the map $\overline{V} \to \St(\Phi, A[X, X\inv])$ given by $(g, [h], u) \mapsto g$ is a map of $\St(\Phi, B)$-sets.
\end{enumerate}
\end{prop}
\begin{proof} By Lemmas~\ref{R1_d}--\ref{R2_0_1} the action of $\St^{\leq 1}(\Phi, B)$ on $\overline{V}$ given by~\eqref{T_1}--\eqref{T_leq0} is well-defined (thanks to~\cref{superfluous-relations} we do not need to verify that it satisfies $\text{R3}_{1,1}^\bot$). On the other hand, the group $\St^{\leq 1}(\Phi, B)$ is isomorphic to $\St(\Phi, B)$ by~\cref{lemma33} (recall that $t = X^{-1}$).
  
The first property can be verified directly using~\eqref{T_1}--\eqref{T_leq0} and the fact that $\St(\Phi, A[X\inv])$ is generated by $x_\alpha(a)$ and $x_\alpha(aX\inv)$ for $\alpha\in\Phi$, $a\in A$. The second property can be verified in a similar fashion. \end{proof}

For a pair $(R, I)$ denote by $\E(\Phi, R, I)$ the {\it relative elementary subgroup} of $\GG(\Phi, R)$, i.\,e.
the image of the relative Steinberg group $\overline{\St}(\Phi, R, I)$ under the homomorphism $\pi \colon \St(\Phi, R) \to \GG(\Phi, R)$.
From the second property of the above proposition we immediately obtain the following group factorization.
\begin{corollary} For $\Phi$ as in the statement of~\cref{lemma33} one has
\[\E(\Phi, A[X\inv] + M[X]) = \E(\Phi, A[X], M[X]) \cdot \E(\Phi, A[X\inv]).\] \end{corollary}

\subsection{Proof of Horrocks theorem} \label{sec:P1glueing}
Before we proceed with the proof of~\cref{thm:P1glueing} let us briefly recall the relevant notation. As before, $A$ denotes an arbitrary local ring with maximal ideal $M$ and $I = M[X, X\inv]$ is an ideal of both the ring $R=A[X, X\inv]$ and its subring $B = A[X\inv] + M[X]$.

Consider the following commutative diagram, in which the map $t$ is obtained from~\cref{lem:lemma32}:
\begin{equation} \label{big-picture} \begin{tikzcd} & \St(\Phi, A[X], M[X]) \arrow{r}[swap]{\mu_{A[X]}} \arrow{d} & \St(\Phi, A[X]) \arrow{dd}{j_+} \\ & \St(\Phi, R, I) \arrow{d}{t} & \\ \St(\Phi, A[X\inv]) \arrow{r}{j_B^-} & \St(\Phi, B) \arrow{r}{j_R} & \St(\Phi, R). \end{tikzcd} \end{equation}
Notice that the subgroup $\overline{G}_M^{\geq 0}$ (see~\cref{df:GM0_geq0}) coincides with the image of $j_+\mu_{A[X]}$.

Our main goal is to prove that the homomorphism $j_B^-$ is injective. 
In order to achieve this we need to show that $\St(\Phi, B)$ acts transitively on the first components of the triples from $\overline{V}_T$. These first components are parameterized by cosets $\overline{G}_M^{\geq 0}/G_M^0$ (cf.~\eqref{VT-def}). It is clear from~\eqref{big-picture} that $\overline{G}_M^{\geq 0}$ lies in the image of $j_R$.
To simplify notation we want to identify $\overline{G}_M^{\geq 0}$ with a subgroup of $\St(\Phi, B)$ by means of $j_R$. Thus, throughout this section we additionally assume that $j_R$ is injective. This assumption is innocent, since we already know from~\cref{thm41} that $j_R$ is injective under the assumptions of~\cref{thm:main}.

\begin{lemma}\label{lem:action} For any $[p, h, u]\in \overline{V}_T$ and any $p_1 \in \overline{G}^{\geq 0}_M$ one has 
\[j_R^{-1}(p_1) \cdot [p, h, u] = [p_1p, h, u].\] \end{lemma}
\begin{proof} 
By our assumption of the injectivity of $j_R$ and the above discussion the preimage $j_R^{-1}(p_1)$ consists of only one element, so the statement of the lemma is unambiguous.

In view of~\cref{thm:Tits} it suffices to verify the assertion of the lemma for the generators $p_1 = z_\alpha(f, \xi)$ of $\overline{G}_M^{\geq 0}$ for all $f \in M[X]$, $\xi\in A[X]$, $\alpha \in \Phi$. We accomplish this by induction on the degree of $\xi$ in $X$. 

Notice that in the base case $\xi = a \in A$ the assertion of the lemma is an immediate consequence of~\eqref{T_leq0}.
Since $x_\beta(X^2f) \in P_\alpha(0)$, an argument similar to the proof of~\cref{R1_d} shows that the assertion of the lemma also holds for $p_1 = z_\alpha(X^2f, aX^{-1})$.

Now let us verify the induction step.
Suppose that the assertion holds for all $p_1 = z_\alpha(f, \xi)$ for which $\xi$ has degree $\leq n$.
Substituting $s := fX$, $\eta := X^{-1}$, $\xi := X\xi$ into~\cref{Zrels}\ref{Z5} we obtain the following equality in $\St(\Phi, B)$:
\begin{multline*}
 z_{\alpha+\beta}(f, X\xi) = x_\alpha(\epsilon Xf) \cdot x_{-\beta}(-X^2 f\xi ) \cdot x_{\beta}(\xi f) \cdot x_{\alpha+\beta}(f) \cdot \\ \cdot z_\alpha(-\epsilon Xf, -\epsilon \xi) \cdot x_{-\alpha}(-\epsilon X\xi^2f) \cdot x_{-\alpha-\beta}(- X^2f \xi^2) \cdot z_{-\beta}(X^2f\xi, -X^{-1}), \end{multline*}
where $\epsilon = N_{\alpha, \beta}$.

From the inductive assumption we obtain that the assertion of the lemma holds for all the factors in the right-hand side and, therefore, also holds for $p_1 = z_\alpha(f, X\xi)$. It is easy to deduce from this that the assertion also holds for $p_1 = z_\alpha(f, X\xi + a) = x_\alpha(-a) \cdot z_\alpha(f, X\xi) \cdot x_\alpha(a)$. 
\end{proof}

\begin{rem} Although we do not need this for our main result, it can be noted that $\overline{V}$ is a left $\St(\Phi, B)$-torsor, i.\,e. the action of $\St(\Phi, B)$ on $\overline{V}$ is both transitive and faithful. The faithfulness follows from the second property of~\cref{prop43} and our assumption that $j_R$ is injective. The transitivity follows from~\cref{lem:action}, the first property of~\cref{prop43} and the following formula, which is a direct consequence of~\cref{lem:Tulenbaev-formula} and~\eqref{T_1}--\eqref{T_leq0}:
\begin{equation*} \langle a, m \rangle^{-1} \cdot \langle aX^{-1}, mX \rangle \cdot [1, 1, 1] = [1, 1, 1+am]. \end{equation*} \end{rem}

\begin{theorem} \label{thm:P1glueing}
 Let $\Phi$ be a root system of type $\rA_{\geq 4}, \rD_{\geq 5}$ or $\rE_{6,7,8}$.
 Assume additionally that the homomorphism $j_R \colon \St(\Phi, B) \to \St(\Phi, A[X, X\inv])$ is injective.  
 Then the homomorphism $j_-$ is injective and the following commutative square is pullback
 \[ \begin{tikzcd} \St(\Phi, A) \arrow{r} \arrow{d} & \St(\Phi, A[X]) \arrow{d}{j_+} \\ \St(\Phi, A[X\inv]) \arrow{r}{j_-} & \St(\Phi, A[X, X\inv]). \end{tikzcd} \] 
\end{theorem}
\begin{proof} Since $j_- = j_R j_B^-$, where $j_B^-$ is as in~\cref{prop43}, it suffices to show that $j_B^-$ is injective.
 Set $v_0 = [1, 1, 1] \in \overline{V}_T$ and suppose $g \in \Ker(j_B^-)$.
 By~\cref{prop43} one has $v_0 = j_B^-(g) \cdot v_0 = [1, g, 1]$, therefore by the definition of $\overline{V}_T$, $g = i_-(\gamma)$ for some $\gamma \in G_M^0$ satisfying $j_-i_-(\gamma) = j_+i_+(\gamma) = 1$.
 Notice that the homomorphism $j_+i_+ = j_-i_-$ is split by the evaluation homomorphism $ev_{X=1}^*$ and therefore is injective. Thus, we conclude that $g = 1$ and that $j_-$ is injective. By symmetry, we also have that $j_+$ is injective (we can swap $X$ with $X^{-1}$ in all statements, including the statement of our assumption that $j_R$ is injective).
 
 Now suppose that $g_+ \in \St(\Phi, A[X])$ and $g_- \in \St(\Phi, A[X\inv])$ are such that $j_-(g_-) = j_+(g_+)$.
 By~\cref{field-injectivity} the image of $j_+(g_+) = j_-(g_-)$ in $\St(\Phi, k[X, X\inv])$ belongs to $\St(\Phi, k)$ and therefore coincides with the image in $\St(\Phi, k)$ of some $g_0 \in \St(\Phi, A)$.
 
 Set $h_+ = g_+\cdot g_0^{-1}$ and $h_- = g_- \cdot g_0^{-1}$.
 It is clear that \[ h_+ \in \overline{\St}(\Phi, A[X], M[X]),\ h_- \in \overline{\St}(\Phi, A[X\inv], M[X\inv])\]
 and that $j_-(h_-) = j_+(h_+).$
 
 From~\cref{prop43} and~\cref{lem:action} we obtain that
 \begin{equation} \nonumber [j_+(h_+)^{-1}, h_-, 1] = j_R^{-1}(j_+(h_+)^{-1}) \cdot [1, h_-, 1] = j_R^{-1}(j_+(h_+)^{-1} \cdot j_-(h_-)) \cdot [1, 1, 1] = [1, 1, 1], \end{equation} therefore, by the definition of~$\overline{V}_T$, $h_- = i_-(\gamma_1)$, $j_+(h_+) = j_+i_+(\gamma_1)$ for some $\gamma_1 \in G_M^0$. Since $j_+$ is injective we obtain that $h_+ = i_+(\gamma_1)$. Thus, we have shown that $g_+$ and $g_-$ are the images of $\gamma_1 \cdot g_0 \in \St(\Phi, A)$ under $i_+$ and $i_-$, respectively.
 \end{proof}
  
\begin{proof}[Proof of~\cref{thm:main}]
 Notice that the assertions of the theorem for $\KO_2(2\ell, -)$ and $\K_2(\rD_\ell, -)$ follow from the assertion for $\St(\rD_\ell, -)$. It is also clear that the latter assertion follows from~\cref{thm41} and~\cref{thm:P1glueing} in the special case when $A$ is a local ring.
 
 Now let $A$ be an arbitrary commutative ring and $\Phi = \rD_\ell$ for $\ell \geq 7$. If $g \in \St(\Phi, A[X])$ is such that its image in $\St(\Phi, A[X, X\inv])$ is trivial, then so is its image in all localizations $\St(\Phi, A_M[X, X\inv])$, where $M$ ranges over the maximal ideals of $A$.
 We denote by $\lambda_{M}$ (resp. $\lambda_{M, -}$) the localization homomorphism $A[X] \to A_M[X]$ (resp. $A[X\inv] \to A_M[X\inv]$) and by $\lambda_M^*$, $\lambda_{M,-}^*$ the corresponding homomorphisms of Steinberg groups.
 By the previous paragraph the images $\lambda^*_M(g)$ in all $\St(\Phi, A_M[X])$ are also trivial. Now by the local-global principle \cite[Theorem~2]{LS17} the element $g$ is trivial as well.

 Now suppose that $g_+ \in \St(\Phi, A[X])$ and $g_- \in \St(\Phi, A[X\inv])$ are such that $j_+(g_+) = j_-(g_-)$. Set $g_0 = ev^*_{X=0}(g_+)$, $h_+ = g_+ \cdot i_+(g_0^{-1})$, $h_- = g_- \cdot i_-(g_0^{-1})$.
 Notice that $h_+ \in \overline{\St}(\Phi, A[X], XA[X])$, moreover, for every maximal ideal $M$ of $A$ the image $\lambda_M^*(h_+)\in \overline{\St}(\Phi, A_M[X], XA_M[X])$ is trivial by~\cref{thm:P1glueing}
  (since $j_+(\lambda^*_M(h_+)) = j_-(\lambda^*_{M,-}(h_-))$).
 Thus, again by the local-global principle~\cite[Theorem~2]{LS17} the element $h_+$ is trivial, therefore $g_+ = i_+(g_0)$.
 Using similar argument one can show that $g_- = i_-(g_0')$ for some $g_0' \in \St(\Phi, A)$. But $g_0'$ must coincide with $g_0$ since $j_+i_+ = j_-i_-$ is injective.
\end{proof}

\printbibliography
\end{document}